\documentclass[preprint,12pt]{elsarticle}

\pdfoutput=1

\usepackage{amssymb,amsmath,amsfonts}
\usepackage{amsthm}

\usepackage{subfig}
\usepackage{algorithm,algorithmic}

\newcommand{\mJ}{\mathbf{J}}
\newcommand{\mQ}{\mathbf{Q}}
\newcommand{\mP}{\mathbf{P}}
\newcommand{\mW}{\mathbf{W}}

\newcommand{\ve}{\mathbf{e}}
\newcommand{\vf}{\mathbf{f}}
\newcommand{\vi}{\mathbf{i}}
\newcommand{\vj}{\mathbf{j}}
\newcommand{\vl}{\mathbf{l}}
\newcommand{\vm}{\mathbf{m}}
\newcommand{\vn}{\mathbf{n}}

\newcommand{\bpi}{\mbox{\boldmath$\pi$}}
\newcommand{\bnu}{\mbox{\boldmath$\nu$}}

\newcommand{\sA}{\mathcal{A}}

\newcommand{\sI}{\mathcal{I}}

\newcommand{\sS}{\mathcal{S}}
\newcommand{\sU}{\mathcal{U}}

\newtheorem{theorem}{Theorem}
\newtheorem{lemma}{Lemma}
\newtheorem{remark}{Remark}
\newtheorem{corollary}{Corollary}

\journal{CMAME}

\begin{document}

\begin{frontmatter}



\title{Sparse Pseudospectral Approximation Method}


\author[label3]{Paul G. Constantine\corref{cor1}}
\ead{paul.constantine@stanford.edu}
\cortext[cor1]{Building 500, Room 501-L, Tel: (650) 723-2330}
\author[label1]{Michael S. Eldred}
\ead{mseldre@sandia.gov}
\author[label1]{Eric T. Phipps}
\ead{etphipp@sandia.gov}
\fntext[label2]{Sandia National Laboratories is a multi-program laboratory managed
and operated by Sandia Corporation, a wholly owned subsidiary of Lockheed Martin Corporation, for the U.S. Department
of Energy's National Nuclear Security Administration under contract DE-AC04-94AL85000.}

\address[label1]{Sandia National Laboratories\fnref{label2}, Albuquerque, NM 87185}
\address[label3]{Stanford University, Stanford, CA 94305}

\begin{abstract}
Multivariate global polynomial approximations -- such as polynomial chaos or stochastic collocation methods -- are now
in widespread use for sensitivity analysis and uncertainty quantification. The pseudospectral variety of these methods
uses a numerical integration rule to approximate the Fourier-type coefficients of a truncated expansion in orthogonal
polynomials. For problems in more than two or three dimensions, a sparse grid numerical integration rule offers
accuracy with a smaller node set compared to tensor product approximation. However, when using a sparse rule to
approximately integrate these coefficients, one often finds unacceptable errors in the coefficients associated with 
higher degree polynomials.

By reexamining Smolyak's algorithm and exploiting the connections between interpolation and projection in tensor
product spaces, we construct a \emph{sparse pseudospectral approximation method} that accurately reproduces the
coefficients for basis functions that naturally correspond to the sparse grid integration rule. The compelling numerical
results show that this is the proper way to use sparse grid integration rules for pseudospectral approximation.
\end{abstract}


\end{frontmatter}


\section{Introduction}

As the power and availability of computers has increased, the profile of simulation in scientific and engineering
endeavors has risen. Computer simulations that model complex physical phenomena now regularly aid in decision making
and design processes. However, the complexity and computational cost of the codes often render them impractical for
design and uncertainty studies, where many runs at different input parameter values are necessary to compute statistics
of interest. In such cases, designers use a relatively small number of high fidelity runs to build cheaper surrogate
models, which are then used for the studies requiring many model evaluations.

It is now common to use a multivariate global polynomial of the input parameters as the surrogate, particularly when
one desires estimates of integrated quantities such as mean and variance of simulation results. Additionally, the
polynomial surrogate is typically much cheaper to evaluate as a function of the input parameters, which allows sampling
and optimization studies at a fraction of the cost. In an uncertainty quantification context -- where the input
parameters often carry the interpretation of random variables -- this polynomial approximation method appears under the
labels \emph{polynomial chaos}~\cite{Xiu02,Ghanem91} or \emph{stochastic collocation}~\cite{Xiu05,Babuska07}, amongst
others.

One of the primary disadvantages of the polynomial methods is the rapid growth in the work required to compute
the approximation as the number of model input parameters increases; this generally limits the
applicability of these methods to models with fewer than ten input parameters. To combat this apparent curse of
dimensionality, many have proposed to use so-called sparse grid methods~\cite{Griebel04}, which deliver comparable
accuracy for some problems using far fewer function evaluations to build the surrogate. The sparse grid is a set of
points in the input parameter space that is the union of carefully chosen tensor product grids. When the tensor grids
are formed from univariate point sets with a nesting property, such as the Chebyshev points, the number of points in the
union of tensor grids is greatly reduced -- although this nesting feature is not necessary for the construction of the
sparse grids. The points in the sparse grid can be used as a numerical integration rule~\cite{Novak96,Nobile08}, where
the weights are linear combinations of weights from the member tensor grids. Alternatively, the interpolating tensor product
Lagrange polynomials constructed on the member tensor grids can be linearly combined in a similar fashion to yield a
polynomial surrogate~\cite{Barthelmann00}, since a linear combination of polynomials is itself a polynomial. 

Another popular polynomial representation employs a multivariate orthogonal polynomial basis. When the coefficients of
a series in this basis are computed by projecting the unknown function onto each basis, the series is a spectral
projection or Fourier series~\cite{Szego39,Dunkl01}; this is also known as the polynomial chaos expansion in an
uncertainty quantification context~\cite{Xiu02,Ghanem91}. The series must be truncated for computation; convergence to
the true function occurs in the mean-squared sense as one adds more basis polynomials. If the integrals in the
projections are approximated with a numerical integration rule, this method is known as a pseudospectral
projection~\cite{Xiu07,LeMaitre20029}. These integral approximations only require the simulation outputs evaluated at
the quadrature points of the input space.

The question that naturally arises is: \emph{Which numerical integration rule is appropriate to approximate the Fourier
coefficients?} Some early attempts used Monte Carlo integration~\cite{Reagan2003545}, but its relative inaccuracies
overwhelm the spectral accuracy of the truncated Fourier series. Other attempts used tensor product Gaussian
quadrature rules, but they do not scale to high dimensional parameter spaces due to the exponential
increase in the number of quadrature points with dimension. The sparse-grid quadrature rules have shown promise for
retaining the spectral accuracy while alleviating the curse of dimensionality. However, in practice this approach
produces unacceptable errors in the coefficients associated with the higher order basis polynomials, which forces a
much stricter truncation than might be expected for the number of function evaluations~\cite{Eldred08}. 

This paper presents a \emph{sparse pseudospectral approximation method} (SPAM) for computing the coefficients of the
truncated Fourier series with the points of the sparse grid integration rule that eliminates the error in the
coefficients associated with higher degree polynomials. This allows the number of terms in the expansion to be
consistent with the number of points in the sparse-grid integration rule. The key is to separately compute the
coefficients of a tensor product polynomial expansion for each tensor grid in the sparse grid. The linear combination
of the tensor weights used to produce the sparse-grid integration weights is then used to linearly combine the
coefficients of each tensor expansion. We show that this method produces a pointwise equivalent polynomial surrogate to
the one constructed from a linear combination of tensor product Lagrange polynomials. Therefore error bounds from that
context can be applied directly. 

Recently, in the context of spectral methods for discretized PDEs, Shen and
coauthors~\cite{Shen10,Shen10-2} proposed and analyzed a closely related sparse spectral approximation using a
hierarchical basis of Chebyshev polynomials; the hierarchical structure results in increased efficiency. Their
computation of the coefficients for the hierarchical basis follow a comparable construction to the one we present.
However, their focus is on approximating the solution to a high-dimensional PDE, as opposed to more general function
approximation. 

The remainder of the paper is structured as follows. In Section \ref{sec:background}, we review the relationship
between Lagrange polynomial interpolation on a set of quadrature points and the pseudospectral approximation for
univariate functions; we then extend this analysis to multivariate tensor product approximation. Section
\ref{sec:background} closes with a review of Smolyak's algorithm. In Section
\ref{sec:method}, we  detail the SPAM for approximating the Fourier coefficients using the function evaluations
at the sparse-grid integration points followed by some interesting analysis results. In Section \ref{sec:numerical}, we
present numerical experiments from (i) a collection of scalar bivariate functions and (ii) an elliptic PDE
model with parameterized coefficients. In each experiment, we compare the approximate Fourier coefficients from the SPAM
with ones computed directly with the sparse grid integration rule. Finally we conclude with a summary and discussion in
Section \ref{sec:conclusions}.

\section{Background and Problem Set-up}
\label{sec:background}

In this section, we briefly review the background necessary to understand the SPAM; in particular, we examine the
relationship between the Lagrange interpolation on a set of Gaussian quadrature points and a pseudospectral
approximation in a basis of orthonormal polynomials. One purpose of this review is to set up the notation, which departs
slightly from the notation in the disparate references. For the orthogonal polynomials, we follow the notation
of~\cite{Gautschi04}.

Consider a multivariate function $f:\sS\rightarrow \mathbb{R}$, where the domain $\sS\subset\mathbb{R}^d$ has a
product structure
\begin{equation}
\sS = \sS_1 \times \cdots \times \sS_d.
\end{equation}
Define a $d$-dimensional point $s=(s_1,\dots,s_d)\in \sS$. The domain is equipped with a positive,
separable weight function $w:\sS\rightarrow\mathbb{R}_+$ where $w(s) = w_1(s_1)\cdots w_d(s_d)$ and 
\begin{equation}
\int_{\sS_k} s_k^a \,w_k(s_k) \,ds_k <\infty,\qquad k=1,\dots,d,\quad a=1,2,\dots
\end{equation}
The $w_k$ are normalized to integrate to 1, which allows the interpretation of $w(s)$ as a probability density
function. In general, we consider functions which are square-integrable on $\sS$, i.e.
\begin{equation}
\int_\sS f(s)^2 \,w(s)\,ds \;<\; \infty.
\end{equation}
Such functions admit a convergent Fourier series in orthonormal basis polynomials,
\begin{equation}
f(s) \;=\; \sum_{i_1=1}^\infty \cdots \sum_{i_d=1}^\infty \hat{f}_{i_1,\dots,i_d}\, \pi_{i_1}(s_1)\cdots\pi_{i_d}(s_d)
\;=\; \sum_{\vi\in\mathbb{N}^d} \hat{f}_{\vi} \, \pi_{\vi}(s),
\end{equation}
where the equality is in the $L_2$ sense, $\vi=(i_1,\dots,i_d)$ is a multi-index, and
\begin{equation}
\label{eq:fourier}
\hat{f}_{\vi} \;=\; \int_\sS f(s)\,\pi_{\vi}(s)\,w(s)\,ds
\end{equation}
is the Fourier coefficient associated with the basis polynomial $\pi_{\vi}(s)$. The $\pi_{i_k}(s_k)$ are univariate
polynomials in $s_k$ of degree $i_k-1$ that are orthonormal with respect to $w_k(s_k)$. In general, a pseudospectral
method uses a numerical integration rule to approximate a subset of the integrals \eqref{eq:fourier}; the remaining
terms are discarded. 

While any square-integrable function admits a convergent
Fourier series in theory, the polynomial approximation methods perform best on a much smaller class of smooth functions;
we will restrict our attention to such function classes when citing appropriate error bounds. Before diving into the
multivariate approximation, we first review the univariate case. 

\subsection{Gaussian Quadrature, Collocation, Pseudospectral Methods}

Consider the problem set-up above with $d=1$. Let $\bpi(s)=[\pi_1(s),\dots,\pi_{n}(s)]^T$ be a vector of the first $n$
polynomials that are orthonormal with respect to the weight function $w(s)$. The components of $\bpi(s)$ satisfy a
recurrence relationship, which we can write in matrix form as
\begin{equation}
\label{eq:mat3term}
s\bpi(s) = \mJ\bpi(s) + \beta_{n+1}\pi_{n+1}(s)\ve_{n},
\end{equation}
where $\ve_n$ is an $n$-vector of zeros with a one in the last entry, and $\mJ$ (known as the \emph{Jacobi matrix}) is
the symmetric, tridiagonal matrix containing the recurrence coefficients,
\begin{equation}
\label{eq:jacobi}
\mJ = 
\begin{bmatrix}
\alpha_1 & \beta_1& & &  \\
\beta_2 & \alpha_2 & \beta_3 & &  \\
 & \ddots & \ddots & \ddots &  \\
 & & \beta_{n-1} & \alpha_{n-1} & \beta_{n}\\
 & & & \beta_{n} & \alpha_{n}
\end{bmatrix}.
\end{equation}
The zeros of $\pi_{n+1}(s)$ generate eigenvalue/eigenvector pairs of $\mJ$ by \eqref{eq:mat3term}, which we write
\begin{equation}
\label{eq:eigJ}
\mJ = \mQ\Lambda\mQ^T,\qquad \Lambda = \mathrm{diag}\big([\lambda_1,\dots,\lambda_{n}]\big),
\end{equation}
where $\mQ(i,j)=\pi_i(\lambda_j)/\|\bpi(\lambda_j)\|$ are the elements of the normalized eigenvectors. The zeros
$\lambda_j$ of $\pi_{n+1}(s)$ are the points of the $n$-point Gaussian quadrature rule for $w(s)$; the quadrature
weights $\nu_j\in\mathbb{R}_+$ are given by
\begin{equation}
\nu_j = \frac{1}{\|\bpi(\lambda_j)\|^2},
\end{equation}
which are the squares of the first component of the $j$th eigenvector. A Gaussian quadrature approximation to the
integral is written
\begin{equation}
\label{eq:quad1d}
\int_{\sS} f(s)\,w(s)\,ds \;\approx\; \sU_q^n(f) \;=\; \sum_{j=1}^{n} f(\lambda_j)\,\nu_j \;=\; \vf^T\bnu.
\end{equation}
The $\sU_q^n$ denotes the linear operation of quadrature applied to $f$; the subscript $q$ is for quadrature. This
notation will be used later when discussing sparse grids. The $n$-vector $\vf$ contains the evaluations of $f(s)$ at
the quadrature points, and the $n$-vector $\bnu$ contains the weights of the quadrature rule. It will be notationally
convenient to define the matrices 
\begin{equation}
\label{eq:PW}
\mP(i,j)=\pi_i(\lambda_j), \qquad \mW=\mathrm{diag}([\sqrt{\nu_1},\dots,\sqrt{\nu_{n}}]),
\end{equation}
and note that the orthogonal matrix of eigenvectors $\mQ$ can be written $\mQ=\mP\mW$.

The spectral collocation approximation of $f(s)$ constructs a Lagrange interpolating polynomial through the Gaussian
quadrature points. Since the points are distinct, the $n-1$ degree interpolating polynomial is unique. We write this
approximation $\sU_l^n(f)$, where the subscript $l$ is for \emph{Lagrange interpolation}, as
\begin{equation}
\label{eq:collocation1d}
f(s) \;\approx\; \sU_l^n(f) \;=\; \sum_{i=1}^{n} f(\lambda_i)\,\ell_i(s) \;=\; \vf^T\vl(s).
\end{equation}
The parameterized vector $\vl(s)$ contains the Lagrange cardinal functions
\begin{equation}
\label{eq:lagrange}
\ell_i(s) = \prod_{j=1,\;j\not=i}^{n} \frac{s-\lambda_j}{\lambda_i-\lambda_j}.
\end{equation}
By construction, the collocation polynomial $\sU_{l}^n(f)$ interpolates $f(s)$ at the Gaussian quadrature points.

The pseudospectral approximation of $f(s)$ is constructed by first truncating its Fourier series at $n$ terms and
approximating each Fourier coefficient with a quadrature rule. If we use the $n$-point Gaussian quadrature, then we can
write the approximation as
\begin{equation}
\label{eq:pseudospec}
f(s) \;\approx\; \sU_{p}^n(f) \;=\; \sum_{i=1}^{n}\hat{f}_i\,\pi_i(s) \;=\;\hat{\vf}^T\bpi(s),
\end{equation}
where $\hat{f}_i$ is the pseudospectral coefficient,
\begin{equation}
\label{eq:pseudocoeff}
\hat{f}_i = \sum_{j=1}^{n} f(\lambda_j)\,\pi_i(\lambda_j)\,\nu_j,
\end{equation}
and the vector $\hat{\vf}$ contains all coefficient approximations; the subscript $p$ on $\sU_p^n$ is for
\emph{pseudospectral}. Note that we have overloaded the notation by defining $\hat{f}_i$ as the pseudospectral
coefficient \eqref{eq:pseudocoeff}, instead of the true Fourier coefficient in \eqref{eq:fourier}. We next state two
lemmas about the relationship between the spectral collocation and pseudospectral approximations for future reference.

\begin{lemma}
\label{lem:dft}
The vector of evaluations of $f$ at the quadrature points $\vf$ is related to the pseudospectral coefficients
$\hat{\vf}$ by
\begin{equation}
\hat{\vf} \;=\; \mQ\mW\vf \;=\; \mP\mW^2\vf. 
\end{equation}
\end{lemma}
 
\begin{proof}
This is easily verified by equation \eqref{eq:pseudocoeff} using the matrices defined in \eqref{eq:PW}. 
\end{proof}

\begin{lemma}
\label{lem:cpequiv}
The pseudospectral approximation $\sU_p^n(f)$ is equal to the spectral collocation approximation $\sU_l^n(f)$ for all
$s\in\sS$.
\end{lemma}

\begin{proof}
By the uniqueness of the Lagrange polynomial interpolation, we can write $\mP\vl(s) = \bpi(s)$. Since $\mP=\mQ\mW^{-1}$,
we have $\vl(s)=\mW\mQ^T\bpi(s)$. Then
\begin{align*}
\sU_l^n(f) &= \vf^T\vl(s)\\
&= \vf^T\mW\mQ^T\bpi(s)\\
&= \hat{\vf}^T\bpi(s)\\
&= \sU_p^n(f),
\end{align*}
as required.
\end{proof}

Lemma \ref{lem:cpequiv} implies that the pseudospectral approximation $\sU_p^n(f)$ interpolates $f(s)$ at
the Gaussian quadrature points. However, the equivalence expressed in Lemma \ref{lem:cpequiv} breaks down in two
important cases. When the number of terms in the orthogonal series is less than the number of points in the quadrature
rules, the orthogonal series representation no longer produces the same polynomial as the Lagrange interpolant. Also,
if a quadrature rule that is not the Gaussian quadrature rule is used to approximate the Fourier coefficients, then the
discrete Fourier transform from Lemma \ref{lem:dft} is no longer valid. The latter situation may occur if an alternative
quadrature rule holds practical advantages over the Gaussian quadrature rule. 

\begin{remark}
We have restricted our attention to orthonormal polynomials and Gaussian quadrature rules for a given weight function.
However, transformations similar to Lemma \ref{lem:dft} apply for Chebyshev polynomials and Clenshaw-Curtis quadrature
rules using a fast Fourier transform. For an insightful discussion of the comparisions between these methods of
integration and approximation, see~\cite{Trefethen08}.
\end{remark}

\subsection{Tensor Product Extensions}

When $d>1$, the above concepts extend directly via a tensor product construction. For a given multi-index
$\vn=(n_1,\dots,n_d)\in\mathbb{N}^d$, it is convenient to define the set of multi-indices 
\begin{equation}
\sI_{\vn} = \{\vi \;:\; \vi\in\mathbb{N}^d,\, 1\leq i_k \leq n_k,\, k=1,\dots,d\}.
\end{equation}
We can use this index set to reference components of the tensor product approximations.

Tensor product Gaussian quadrature rules are constructed by taking cross products of univariate Gaussian
quadrature rules. For multi-index $\vn$, let $G_{\vn}$ be the set of $d$-variate Gaussian quadrature points, 
\begin{equation}
G_{\vn} \;=\; \{\lambda_{\vi}=(\lambda_{i_1},\dots,\lambda_{i_d}) \;:\; \vi\in\sI_{\vn} \},
\end{equation}
where the points $\lambda_{i_k}$ with $i_k=1,\dots,n_k$ are the univariate quadrature points for $w_k(s_k)$. 
The associated quadrature weights $W_\vn$ are given by
\begin{equation}
W_{\vn} \;=\; \big\{\,\nu_{\vi}=\nu_{i_1}\cdots\nu_{i_d} \;:\; \vi\in\sI_{\vn} \,\big\}.
\end{equation}
In words, the tensor product quadrature weights are products of the univariate weights. To approximate the integral of
$f(s)$, compute
\begin{align*}
\int_\sS f(s)\,w(s)\,ds
&\approx
(\sU_q^{n_1}\otimes\cdots\otimes\sU_q^{n_d})(f)\\
&=
\sum_{i_1=1}^{n_1}\cdots\sum_{i_d=1}^{n_d} f(\lambda_{i_1},\dots,\lambda_{i_d})\, \nu_{i_1}\cdots\nu_{i_d}\\
&= 
\sum_{\vi\in\sI_{\vn}} f(\lambda_{\vi})\,\nu_{\vi}\\
&=
\vf_{\vn}^T\,\bnu_{\vn}
\end{align*}
where $\vf_{\vn}$ is the vector of function evaluations at the tensor grid of quadrature points, and $\bnu_{\vn}$ is a
vector of the tensor product quadrature weights.

The spectral collocation approximation on the points $G_{\vn}$ uses a basis of product-type Lagrange cardinal functions.
Define the vector of these basis polynomials by
\begin{equation}
\vl_{\vn}(s) = \vl_{n_1}(s_1) \otimes \cdots \otimes \vl_{n_d}(s_d),
\end{equation}
where $\vl_{n_k}(s_k)$ is a vector of the univariate Lagrange cardinal functions constructed on the univariate
quadrature rule defined by $\lambda_{i_k}$; see \eqref{eq:lagrange}. Then the tensor product spectral collocation
approximation for the multi-index $\vn$ is given by
\begin{align}
\label{eq:tenscoll}
f(s) 
&\approx 
(\sU_l^{n_1}\otimes\cdots\otimes\sU_l^{n_d})(f)\\ 
&= 
\sum_{i_1=1}^{n_1}\cdots\sum_{i_d=1}^{n_d} f(\lambda_{i_1},\dots,\lambda_{i_d})\,\ell_{i_1}(s_1)\cdots
\ell_{i_d}(s_d)\\ 
&=
\sum_{\vi\in\sI_{\vn}} f(\lambda_{\vi})\, \ell_{\vi}(s)\\ 
&=
\vf_{\vn}^T\,\vl_{\vn}(s).
\end{align}
The tensor product pseudospectral approximation uses a product type multivariate orthonormal polynomial basis, which is
simply a Kronecker product of the univariate orthonormal polynomials. For a multi-index $\vn$, let $\bpi_{n_k}(s_k)$ be
the vector of univariate polynomials that are orthonormal with respect to $w_k(s_k)$ for $k=1,\dots,d$. Then the vector
\begin{equation}
\bpi_{\vn}(s) = \bpi_{n_1}(s_1)\otimes\cdots\otimes\bpi_{n_d}(s_d)
\end{equation}
contains polynomials that are orthonormal with respect to $w(s)$; we can uniquely reference a component of the vector
$\bpi_{\vn}(s)$ by $\pi_{\vi}(s)$ with $\vi\in\sI_{\vn}$. The tensor product pseudospectral approximation for the
multi-index $\vn$ is given by
\begin{align}
f(s) 
&\approx\; 
(\sU_p^{n_1}\otimes\cdots\otimes\sU_p^{n_d})(f)\\ 
&=
\sum_{i_1=1}^{n_1}\cdots\sum_{i_d=1}^{n_d} \hat{f}_{i_1,\dots,i_d}\, \pi_{i_1}(s_1)\cdots\pi_{i_d}(s_d)\\
&=
\sum_{\vi\in\sI_{\vn}} \hat{f}_{\vi}\, \pi_{\vi}(s)\\ 
&=
\hat{\vf}_{\vn}^T\,\bpi_{\vn}(s), \label{eq:tenspseudo}
\end{align}
where $\hat{\vf}_{\vn}$ is the vector of pseudospectral coefficients
\begin{align}
\hat{f}_{\vi} 
&= 
\sum_{j_1=1}^{n_1}\cdots\sum_{j_d=1}^{n_d}
f(\lambda_{j_1},\dots,\lambda_{j_d})\,\pi_{i_1}(\lambda_{j_1})\cdots\pi_{i_d}(\lambda_{j_d})
\,\nu_{j_1}\cdots\nu_{j_d}\\ 
&=
\sum_{\vj\in\sI_{\vn}} f(\lambda_{\vj}) \,\pi_{\vi}(\lambda_{\vj})\, \nu_{\vj}.
\end{align}
The extensions of Lemmas \ref{lem:dft} and \ref{lem:cpequiv} are then straightforward. For the multi-index $\vn$,
define the matrices
\begin{equation}
\mQ = \mQ_{n_1}\otimes\cdots\otimes\mQ_{n_d}, \qquad
\mP = \mP_{n_1}\otimes\cdots\otimes\mP_{n_d}, \qquad
\mW = \mW_{n_1}\otimes\cdots\otimes\mW_{n_d}.
\end{equation}
The proofs of Lemmas \ref{lem:dft} and \ref{lem:cpequiv} hold with 
\begin{equation}
\vf=\vf_{\vn},\qquad 
\hat{\vf}=\hat{\vf}_{\vn}, \qquad
\vl(s)=\vl_{\vn}(s),\qquad
\bpi(s)=\bpi_{\vn}(s).
\end{equation}
This is easily verified by employing the mixed product property of Kronecker products. In words, we have that the
Lagrange interpolant constructed on a tensor product of Gaussian quadrature points (i.e., tensor product collocation)
produces the same polynomial approximation as a truncated Fourier expansion with a tensor product basis, where the
coefficients are computed with the tensor product Gaussian quadrature rule (i.e., tensor product pseudospectral). This
equivalence occurs when the number of quadrature points in each variable is equal to the number of univariate basis
polynomials in each variable; in other words, the number of points is the maximum degree plus one in each variable.
 
\subsection{Smolyak's Algorithm and Sparse Grids}

The inescapable challenge for tensor product approximation is the exponential increase in the work required to compute
the approximation as the dimension increases. An $n$-point quadrature rule in each of $d$ dimensions uses $n^d$ function
evaluations. Thus, tensor product approximation quickly becomes infeasible beyond a handful of dimensions. Smolyak's
algorithm~\cite{Smolyak63} attempts to alleviate this curse of dimensionality while retaining integration and
interpolation accuracy for certain classes of functions.

The majority of sparse grid applications in the literature rely on Smolyak's algorithm. The most common
derivation starts by defining a linear operation (e.g., integration, interpolation, or projection) on a univariate
function. We can generalize the notation used in \eqref{eq:quad1d}, \eqref{eq:collocation1d}, and \eqref{eq:pseudospec}
by writing the linear operation as $\sU^m(f)$. However, it is common to reinterpret the parameter $m$ in this context as
a choice for how the number of points grows as $m$ is incremented. For example, $n_m=m$ for $m>0$ would correspond
to \eqref{eq:quad1d}, \eqref{eq:collocation1d}, and \eqref{eq:pseudospec}. Another common growth relationship
is
\begin{equation}
\label{eq:growth}
n_m = 2^m-1,\quad m\geq 1. 
\end{equation}
Such a relationship is useful when the quadrature/interpolation point sets are nested, i.e., the points of the
$n$-point rule are a subset of the points in the $2n+1$ rule. This notably occurs for rules based on (i) Chebyshev
points~\cite{Gautschi04}, (ii) Gauss-Patterson\footnote{The Gauss-Patterson rules contain a specific pattern of
nesting that is not applicable for arbitrary $n$. See the reference for further details.} quadrature
formulas~\cite{Patterson68}, or (iii) equidistant points. In the case of a closed region of interpolation/integration,
one may include and reuse the endpoints of the interval in the sequence of approximations; see for example the popular
Clenshaw-Curtis integration rules~\cite{Clenshaw60}. Nested point sets can greatly increase efficiency if $f(s)$ is
very expensive.

Define $|\vm|=m_1+\cdots+m_d$. Given a univariate linear operator, Smolyak's method can be written
\begin{equation}
\sA = \sum_{\vm\in\sI}   c(\vm) \, (\sU^{m_1}\otimes\cdots\otimes\sU^{m_d}).
\end{equation} 
In the standard formulation~\cite{Barthelmann00,Novak96}, the set of admissible multi-indices $\sI$ is
\begin{equation}
\label{eq:indexset}
\sI \;=\; \big\{\, \vm\in\mathbb{N}^d \;:\; l+1\leq |\vm| \leq l+d \,\big\}
\end{equation} 
for a given level parameter $l$. In this case, the coefficients $c(\vm)$ are
\begin{equation}
\label{eq:smolyakc}
c(\vm) \;=\; (-1)^{l+d-|\vm|} \, {d-1 \choose l+d-|\vm|}.
\end{equation}
However, adaptive and anisotropic versions of Smolyak's algorithm may contain different choices for $\sI$ and $c(\vm)$;
such variations are useful if a function's variability can be primarily attributed to a subset of the inputs.
See~\cite{Nobile08-2,Gerstner98} for details on such methods.

For our purposes, it is sufficient to note that Smolyak's algorithm amounts to a linear combination of tensor product
operations. The specific tensor products are chosen so that no constituent tensor grid contains too many nodes. In the
case of nested univariate rules, a node may be common to many tensor products. In practice, one may structure the
computation to evaluate the function once per node in the union of tensor product grids -- as opposed to once per node
per tensor grid. This greatly simplifies the sparse grid integration, which can be written as a set of nodes and
weights. If a node is common to multiple constituent tensor grids, then its corresponding weight is computed as a linear
combination of the tensor grid weights; the coefficients of the linear combination are exactly $c(\vm)$. It is worth
noting that the weights of a sparse grid rule can be negative, which precludes its use as a positive definite weighted
inner product.

\section{Sparse Pseudospectral Approximation Method}
\label{sec:method}

In practice, one may wish to take advantage of the relatively small number of points in the sparse grid quadrature rule
when computing a pseudospectral approximation. This is often done by first 
truncating the Fourier series representation of $f(s)$ (see \eqref{eq:fourier}), and then approximating its spectral
coefficients with a sparse grid quadrature rule. Unfortunately, choosing the parameters of the sparse grid rule that
will accurately approximate the integral formulation of the Fourier coefficient is not straightforward. The is because
-- in constrast to tensor product approximation -- the Lagrange interpolating polynomial is \emph{not}
equivalent to a truncated pseudospectal approximation with sparse grid integration, where the number of basis
polynomials is equal to the number of points in the quadrature rule. The general wisdom has been to truncate
conservatively for a sparse grid quadrature rule constrained by a computational budget; such heuristics become
more complicated when anisotropic sparse grid rules are used.

The SPAM approaches this problem from a different perspective; it is merely the proper application of Smolyak's
algorithm to the tensor product pseudospectral projection. We take advantage of the equivalence between tensor product
pseudospectral and spectral collocation approximations to construct spectral approximations that naturally correspond
to a given sparse grid quadrature rule. In essence, since the sparse grid quadrature rule is constructed by taking
linear combinations of tensor product quadrature rules, we can take the same linear combination of tensor product
pseudospectral expansions to produce an approximation in a basis of multivariate orthogonal polynomials; a linear
combination of expansions can be easily computed by linearly combining the pseudospectral coefficients corresponding to
the same basis polynomial. Each tensor product pseudospectral expansion is simply a transformation from the Lagrange
basis using Lemma \ref{lem:dft}. In the numerical examples of Section \ref{sec:numerical}, we show compelling evidence
that this procedure is superior to directly applying the sparse grid quadrature rule to the integral formulation of the
Fourier coefficients.

More precisely, let $\sI$ and $c(\vm)$ be the admissible index set and coefficient function for a given sparse grid
quadrature rule. Then the sparse pseudospectral approximation is given by
\begin{align}
f(s) 
&\approx 
\sA_p(f)\\
&=
\sum_{\vm\in\sI} \,c(\vm)\, (\sU_p^{m_1}\otimes\cdots\otimes\sU_p^{m_d})(f)\\
&=
\sum_{\vm\in\sI} \,c(\vm)\, \hat{\vf}_{\vm}^T\,\bpi_{\vm}(s)
\end{align}
where $\hat{\vf}_{\vm}$ and $\bpi_{\vm}(s)$ are defined as in \eqref{eq:tenspseudo}. In practice, we linearly combine
the coefficients corresponding to common basis polynomials. With a slight abuse of notation, let $\{\bpi(s)\}$ be
the set of basis polynomials corresponding to a vector $\bpi(s)$; the common basis set for $\sA_p(f)$ is defined by
\begin{equation}
\label{eq:spambasis}
\Pi = \bigcup_{\vm\in\sI}\, \{\bpi_{\vm}(s)\}.
\end{equation}
Then we can write 
\begin{equation}
\label{eq:rearrange}
\sA_p(f) = \sum_{\pi(s)\in\Pi} \hat{f}_{\pi}\,\pi(s).
\end{equation}
The coefficient corresponding to $\pi(s)$ is given by
\begin{equation}
\label{eq:fpi}
\hat{f}_\pi = \sum_{\vm\in\sI} \,c(\vm)\, \hat{f}_{\vi,\vm},
\end{equation}
where
\begin{equation}
\label{eq:spamc}
\hat{f}_{\vi,\vm} =
\left\{
\begin{array}{cl}
\hat{f}_{\vi} & \mbox{ if $\pi(s)=\pi_{\vi}(s)$ with $\vi\in\sI_{\vm}$,}\\
0 & \mbox{ otherwise.}
\end{array}
\right.
\end{equation}
In words, \eqref{eq:rearrange} simply rearranges the terms in the sum so that each polynomial basis appears only once.
The next theorem allows us to apply existing analysis results for sparse grid interpolation schemes to the SPAM.

\begin{theorem}
Under the conditions of Lemma \ref{lem:cpequiv}, the sparse pseudospectral approximation $\sA_p(f)$ is point-wise
equivalent to the sparse grid interpolation approximation.
\end{theorem}

\begin{proof}
Using the tensor product version of Lemma \ref{lem:cpequiv}, we can write
\begin{align*}
\sA_p(f) &= \sum_{\vm\in\sI} \,c(\vm)\, \hat{\vf}_{\vm}^T\, \bpi_{\vm}(s)\\
 &= \sum_{\vm\in\sI} \,c(\vm)\, \vf_{\vm}^T\, \vl_{\vm}(s),
\end{align*}
where $\vf_{\vm}$ and $\vl_{\vm}(s)$ are defined in \eqref{eq:tenscoll}. This completes the proof.
\end{proof}

As a result of this theorem, all of the error analysis for sparse grid collocation and interpolation methods applies
directly to the sparse pseudospectral approximation. We refer the interested reader to
references~\cite{Barthelmann00,Novak96,Griebel04} for such details. Next, we prove an interesting fact about the mean
of $\sA_p(f)$.

\begin{corollary}
The mean of the sparse pseudospectral approximation $\sA_p(f)$ is equal to the mean of $f(s)$ approximated with the
associated sparse grid quadrature rule. 
\end{corollary}

\begin{proof}
By orthogonality, the mean of a polynomial expanded in an orthonormal basis is equal to the coefficient of the zero
degree term, which is 1. Define $\hat{f}_{\mathbf{1}}$ to be the coefficient of the constant term in $\sA_p(f)$. The
constant term also appears in each constituent tensor product pseudospectral approximation; denote this by
$\hat{f}_{\mathbf{1},\vm}$ for the multi-index $\vm$. Therefore, by \eqref{eq:spamc},
\begin{align*}
\hat{f}_{\mathbf{1}} &= \sum_{\vm\in\sI} \,c(\vm)\, \hat{f}_{\mathbf{1},\vm}\\
&= \sum_{\vm\in\sI} \,c(\vm)\, (\sU_q^{m_1}\otimes\cdots\otimes\sU_q^{m_d})(f),
\end{align*}
which is exactly the definition of sparse grid integration.
\end{proof}

\subsection{Discrete Orthogonality}
\label{sec:discorth}
We will see in the numerical results in the next section that -- across all test cases -- the pseudospectral
coefficients corresponding to the higher order polynomials are inaccurate when computed directly with the sparse grid
integration rule. This occurs because the higher order basis functions are not
orthonormal with respect to the sparse grid quadrature rule. However, when the integrations are performed using the
SPAM, the basis polynomials are orthonormal. This becomes apparent by looking at the SPAM coefficients for each basis
polynomial in the set $\Pi$ from \eqref{eq:spambasis}. 

\begin{theorem}
\label{thm:ortho}
Let $f(s)=\phi(s)$ for some $\phi(s)\in\Pi$ from \eqref{eq:spambasis}. Then
\begin{equation}
\hat{f}_{\pi} = 
\left\{
\begin{array}{cl}
1 & \mbox{ if $\pi(s)=\phi(s)$,}\\
0 & \mbox{ otherwise,}
\end{array}
\right.
\end{equation}
where $\hat{f}_\pi$ is from \eqref{eq:fpi}.
\end{theorem}

\begin{proof}
Using Proposition 3 from~\cite{Barthelmann00}, we have $\sA_p(f) = f$ for $f=\phi\in\Pi$, which implies that $\sA_p$ is
a projector for the polynomial space defined by $\mathrm{span}\,(\Pi)$. Noting that the elements of $\Pi$ are linearly
independent completes the proof.
\end{proof}

Figure \ref{fig:orth-a} numerically verifies the orthonormality of the elements of $\Pi$ using the SPAM; Figure
\ref{fig:orth-b} demonstrates the loss of orthonormality for the higher order elements of $\Pi$ for a discrete
inner product defined by the points and weights of the sparse grid integration rule.
We use $d=2$ and $l=4$, and we order the basis polynomials by their degree. Notice that some of the lower order basis
polynomials are orthonormal with respect to a discrete norm defined by the sparse grid quadrature rule. This is due to
the degree of exactness of the sparse grid quadrature rule; see~\cite{Novak99} for more details. 

\begin{figure}%
\centering
\subfloat[SPAM]{\label{fig:orth-a}
\includegraphics[scale=0.35]{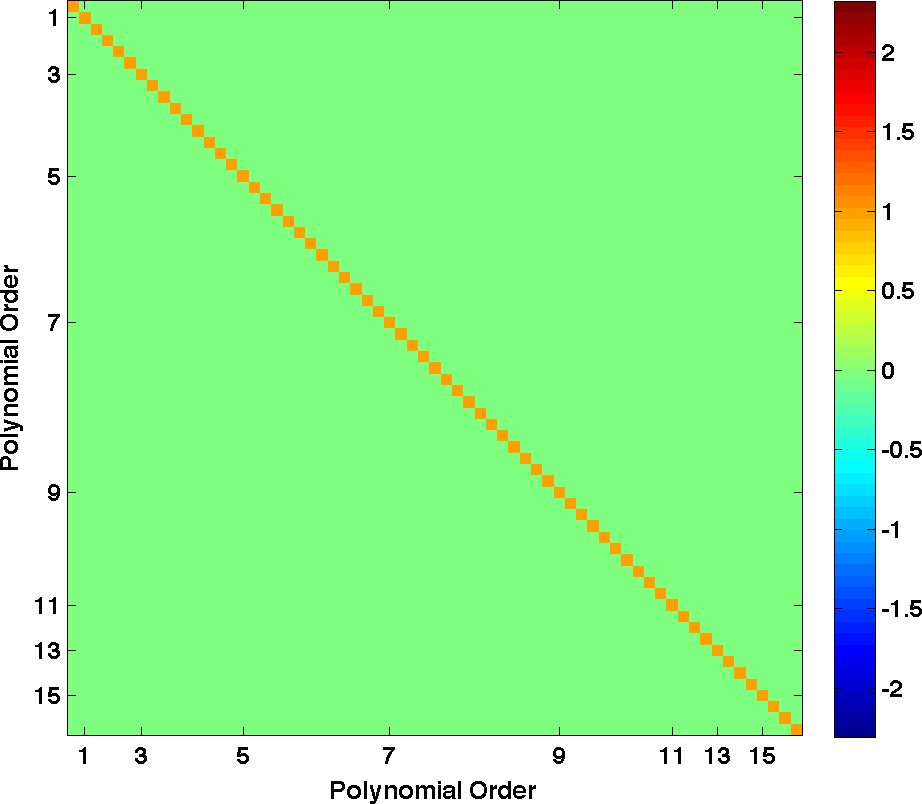}
}
\quad
\subfloat[Sparse grid integration]{\label{fig:orth-b}
\includegraphics[scale=0.35]{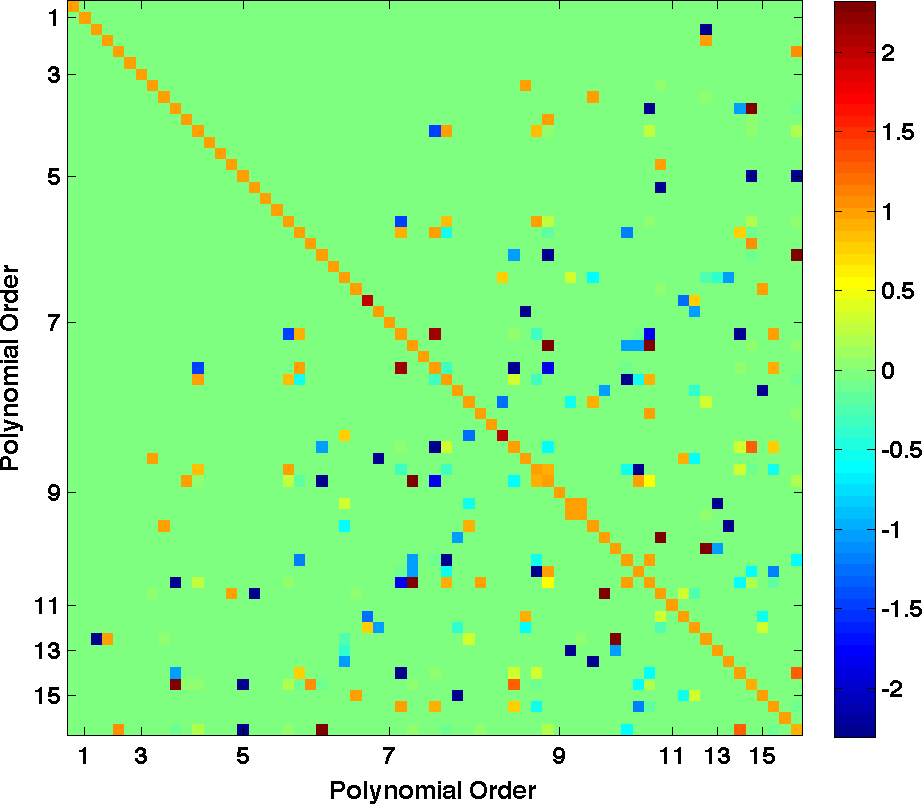}
}
\caption{Orthonormality of the elements in $\Pi$ from \eqref{eq:spambasis} using SPAM or directly approximated with the
sparse grid integration rule with dimension $d=2$ and level $l=4$. The sparse grid was built from univariate
Gauss-Legendre quadrature rules with growth rule \eqref{eq:growth}.}
\end{figure}

\section{Numerical Experiments}
\label{sec:numerical}

In the following numerical experiments, we compare the coefficients computed with the SPAM to direct approximation of
the Fourier coefficients with the corresponding sparse grid quadrature rule. To make the comparison fair, we apply the
sparse grid rule directly to each coefficient corresponding to the basis set \eqref{eq:spambasis} for the sparse
pseudospectral approximation. We construct each sparse grid rule using (i) univariate non-nested Gauss-Legendre
quadrature points for a uniform weight function on the square $[-1,1]^2$, (ii) $n_m$ defined as in
\eqref{eq:growth}, and (iii) $\sI$ and $c(\vm)$ defined as in \eqref{eq:indexset} and \eqref{eq:smolyakc}. The choice
of the uniform weight function implies the $\pi_i(s)$ are the normalized Legendre polynomials for the pseudospectral
approximation. For all experiments, we compute the largest feasible tensor product pseudospectral approximation
and call the resulting coefficients the truth. In all cases, the apparent decay in the tensor product pseudospectral
coefficients assures us that we have used a sufficiently high order approximation to bestow the title \emph{truth}. 

\subsection{Five Bivariate functions}
In the first experiment, we compare both methods on five bivariate functions; see Table \ref{tab:testfuncs}. For each
function, we compute a tensor product pseudospectral approximation of order 255 in each variable -- 65,536 total
quadrature points. We plot the log of the magnitude of the pseudospectral coefficients with a surface plot to visually
observe their decay. We then plot the log of the magnitude of the sparse pseudospectral coefficients corresponding to a
level $l=7$ sparse grid compared to the same sparse grid approximation of the Fourier coefficients.

\begin{table}
\begin{center}
\begin{tabular}{|c|c|}
\hline
\# & $f(s_1,s_2)$\\
\hline
1 &$s_1^{10}s_2^{10}$\\
2 &$e^{s_1+s_2}$\\
3 &$\sin(5(s_1-0.5)) + \cos(3(s_2-1))$\\
4 &$1/(2+16(s_1-0.1)^2+25(s_2+0.1)^2)$\\
5 &$(|s_1-0.2|+|s_2+0.2|)^3$\\
\hline
\end{tabular}
\end{center}
\caption{The five bivariate test functions.}
\label{tab:testfuncs}
\end{table}

With a level 7 grid, sparse pseudospectral approximation contains a maximum univariate degree of 129. For each test
function, the corresponding set of figures contains
\begin{enumerate}
  \item[(i)] the tensor product pseudospectral coefficients up to maximum univariate degree 100,
  \item[(ii)] the SPAM coefficients up to maximum univariate degree 100,
  \item[(iii)] the sparse grid integration approximation of the Fourier coefficients up to maximum univariate degree
  100,
  \item[(iv)] a summary plot with coefficients up to univariate degree 129 ordered by total order.
\end{enumerate}
As a general comment, we see that the sparse grid integration produces largely incorrect values for coefficients
associated with higher degree polynomials. More specific comments for the individual test functions are as follows:
\begin{enumerate}
  \item $s_1^{10}s_2^{10}$: This function evaluates the performance of the methods on a monomial. We know that
  coefficients associated with polynomials of degree greater than 10 in either $s_1$ or $s_2$ should be zero by
  orthogonality. Additionally, since the monomial is an even function over the domain with a symmetric weight function,
  the coefficients corresponding to odd degree polynomials in either variable ought to be zero. This is verified in the
  tensor product pseudospectral coefficients and respected by the SPAM coefficients. However, the direct sparse
  integration produces non-zero values for coefficients that should be zero. See Figure \ref{fig:bv1}.
  \item $e^{s_1+s_2}$: This function is analytic in both variables with rapid decay of the Fourier coefficients. Observe
  that the direct sparse integration yields large values for coefficients corresponding to the higher order polynomials.
  See Figure \ref{fig:bv2}.
  \item $\sin(5(s_1-0.5)) + \cos(3(s_2-1))$: In the language of the ANOVA decomposition~\cite{Liu06}, this function has
  only main effects. Thus, the Fourier coefficients for polynomials with mixed terms corresponding to
  interaction effects should be zero. Again, this is apparent in the tensor product pseudospectral coefficients, and it
  is respected by the SPAM coefficients. The direct sparse integration, however, produces non-zero values for
  coefficients of the mixed polynomials; see Figure \ref{fig:bv3}.
  \item $1/(2+16(s_1-0.1)^2+25(s_2+0.1)^2)$: The pseudospectral coefficients of this rational function decay relatively
  slowly; notice it needs up to degree 40 polynomials in each variable to reach numerical precision, according to the
  tensor product pseudospectral coefficients. The SPAM coefficients do a much better job capturing the true decay of the
  Fourier coefficients than the direct integration method, which does not appear to decay at all.
  See Figure \ref{fig:bv4}.
  \item $(|s_1-0.2|+|s_2+0.2|)^3$: This function has discontinuous first derivatives, so we expect only first order
  algebraic convergence of its Fourier coefficients; on a log scale they decay very slowly. However, the
  interaction effects disappear after degree three in either variable. Again, this is visible in the tensor product
  pseudospectral coefficients and respected by the SPAM coefficients, but the direct sparse grid integration produces
  non-zero values for coefficients that ought to be zero. See Figure \ref{fig:bv5}.
\end{enumerate}
In general, we find that the SPAM coefficients are significantly more accurate than the direct application of the sparse
grid integration rules to the Fourier coefficients. This observation is somewhat counterintuitive. One may expect that
the sparse grid rule, by evaluating the product of the function times the basis polynomial at more locations, would yield a
more accurate approximation. But this is clearly not the case for these examples. The decreased accuracy in the 
coefficients computed with the sparse grid integration rule is a result of the nonorthogonality of the basis polynomials
with respect to a discrete inner product defined by the sparse grid integration rule; see Section \ref{sec:discorth}.

In Figure \ref{fig:err-c}, we plot the
decay of the truncation error for the sparse approximations as the level increases. We approximate the truncation error
by the sum of squares of the neglected coefficients from the tensor product expansion. Since both approximations use
the same basis sets, this approximate truncation error is identical. In Figures \ref{fig:err-a} and \ref{fig:err-b} we
plot the decay in the error of the approximated coefficients as the level increases; the error in the coefficients is
squared and summed. We see that the error in the SPAM coefficients decays roughly like the truncation error, while the
error in the direct sparse grid integration does not decay. Of course, this summary plot ignores what is most visible
in Figures \ref{fig:bv1}-\ref{fig:bv5}, which is that some of the coefficients associated with lower order polynomials
may be approximated well; it is the coefficients of the higher order terms that contain most of the error.

\begin{figure}%
\centering
\subfloat[Tensor $100\times 100$]{
\includegraphics[scale=0.35]{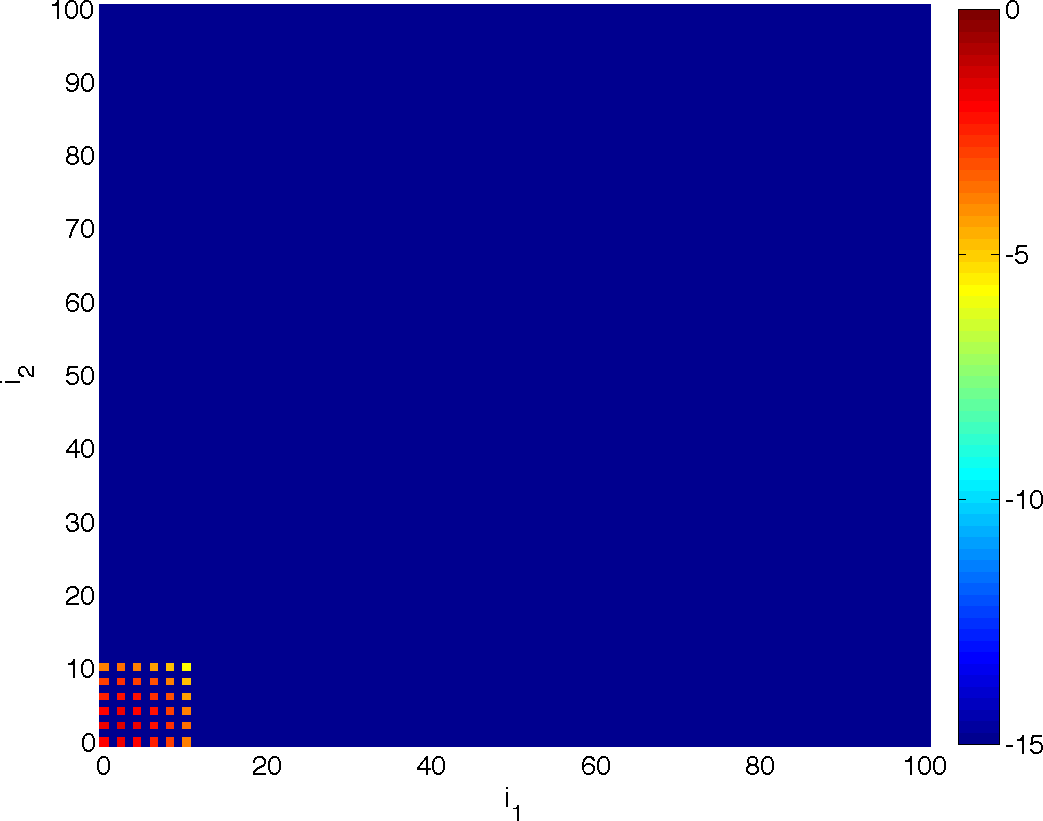}
}
\quad
\subfloat[Coefficient decay]{
\includegraphics[scale=0.35]{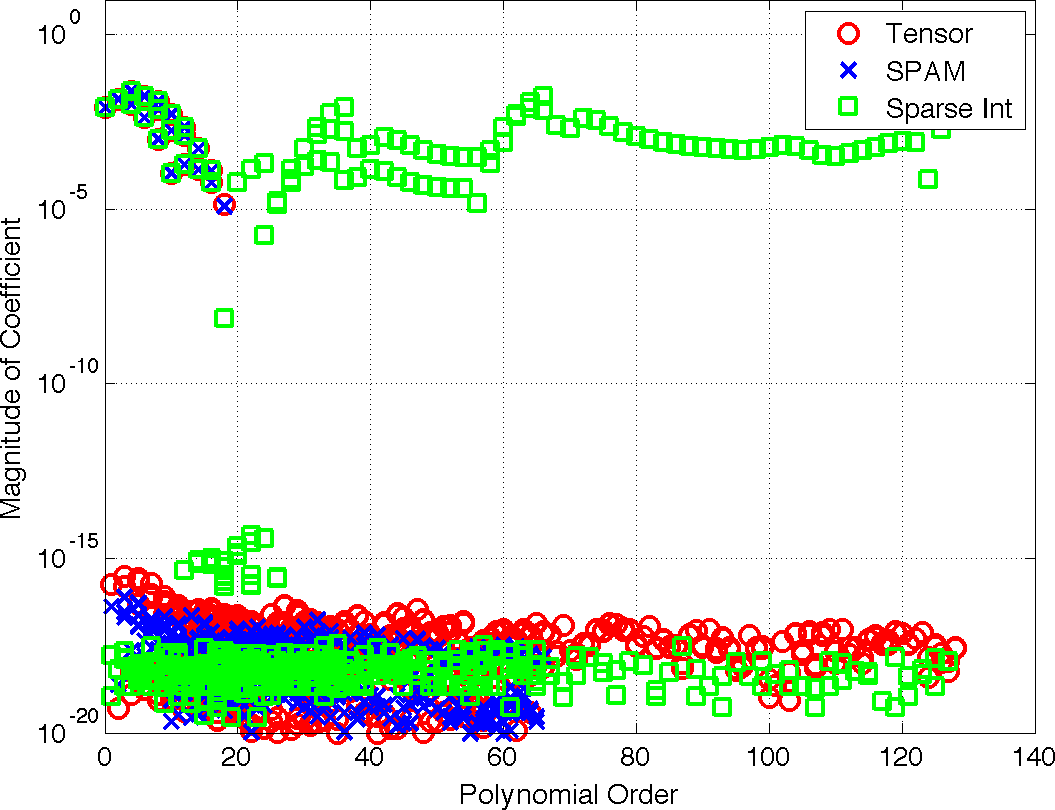}
}\\
\subfloat[SPAM]{
\includegraphics[scale=0.35]{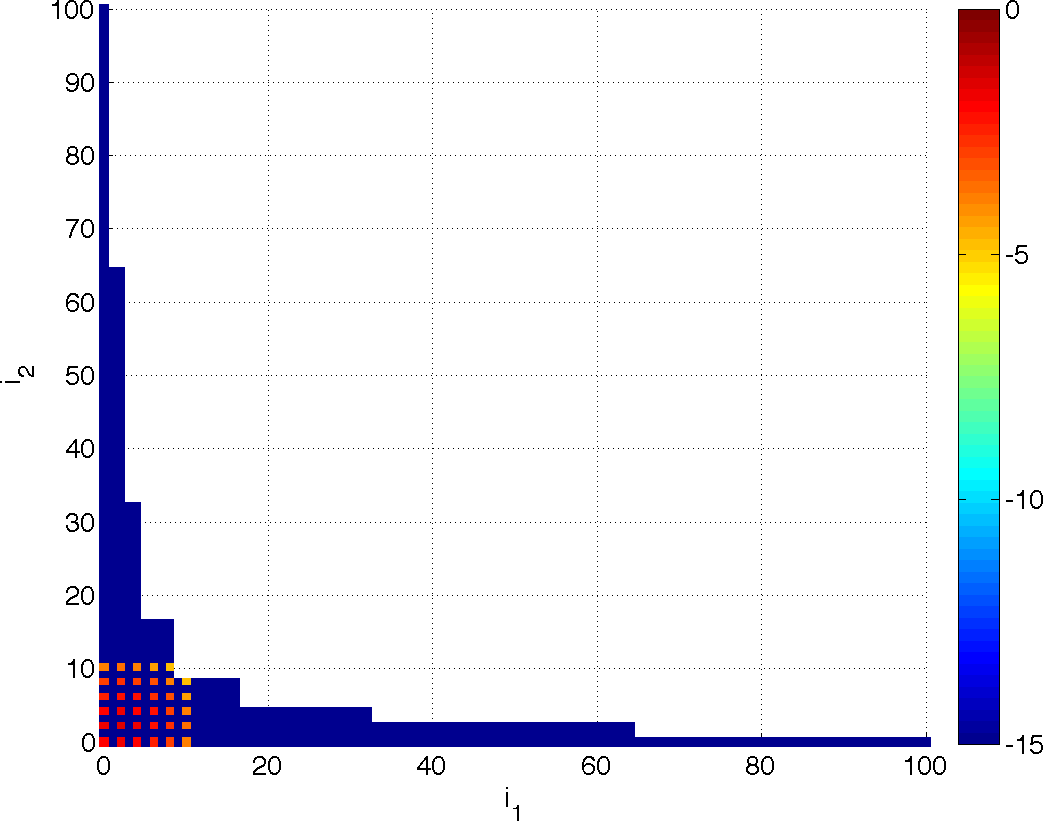}
}
\quad
\subfloat[Sparse grid integration]{
\includegraphics[scale=0.35]{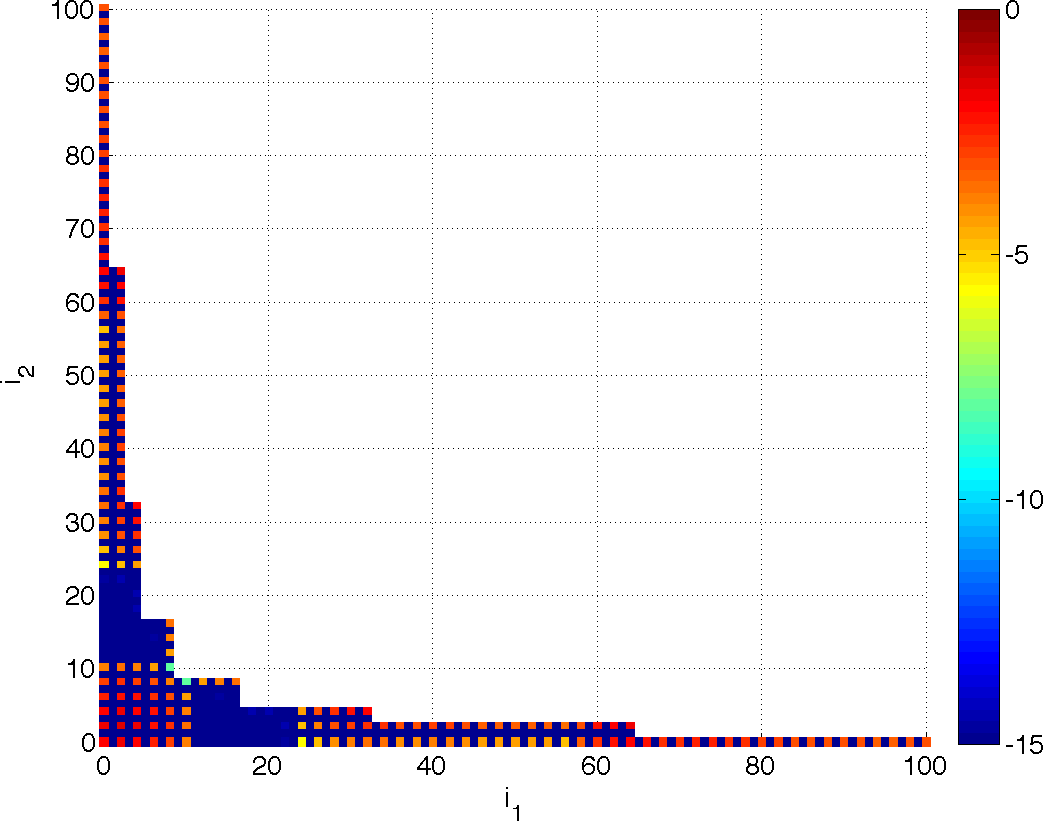}
}
\caption{Fourier coefficient approximations for $s_1^{10}s_2^{10}$.}
\label{fig:bv1}
\end{figure}

\begin{figure}%
\centering
\subfloat[Tensor $100\times 100$]{
\includegraphics[scale=0.35]{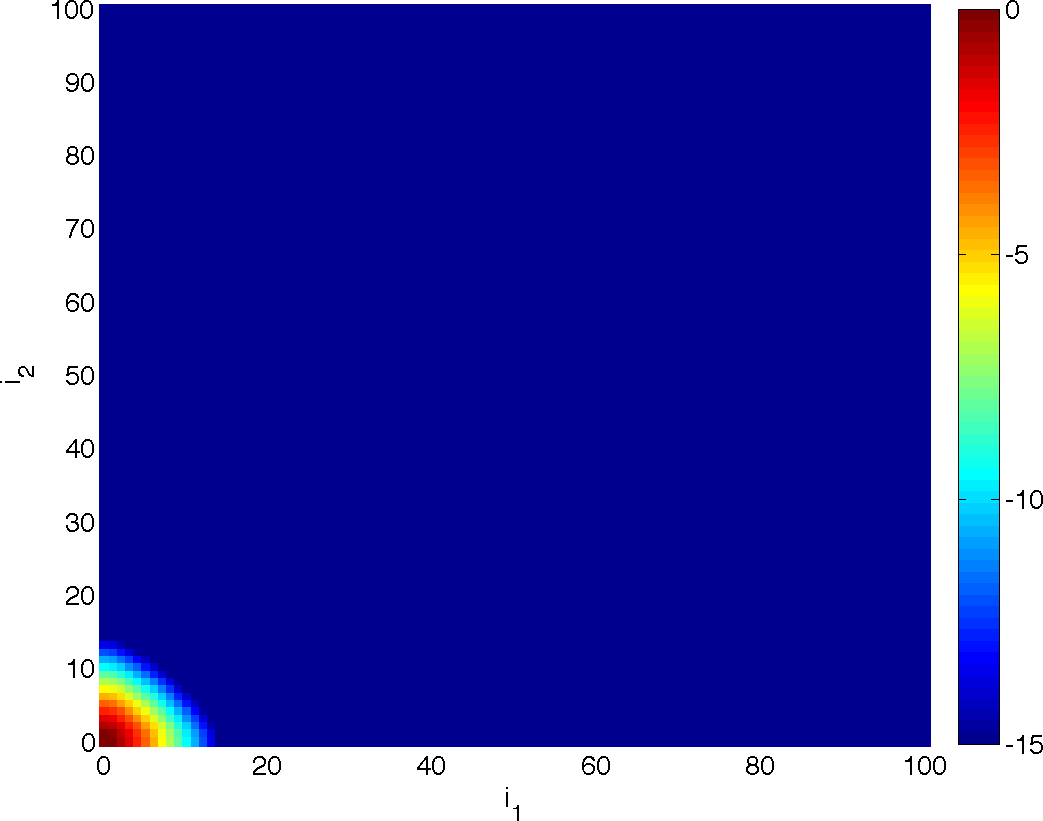}
}
\quad
\subfloat[Coefficient decay]{
\includegraphics[scale=0.35]{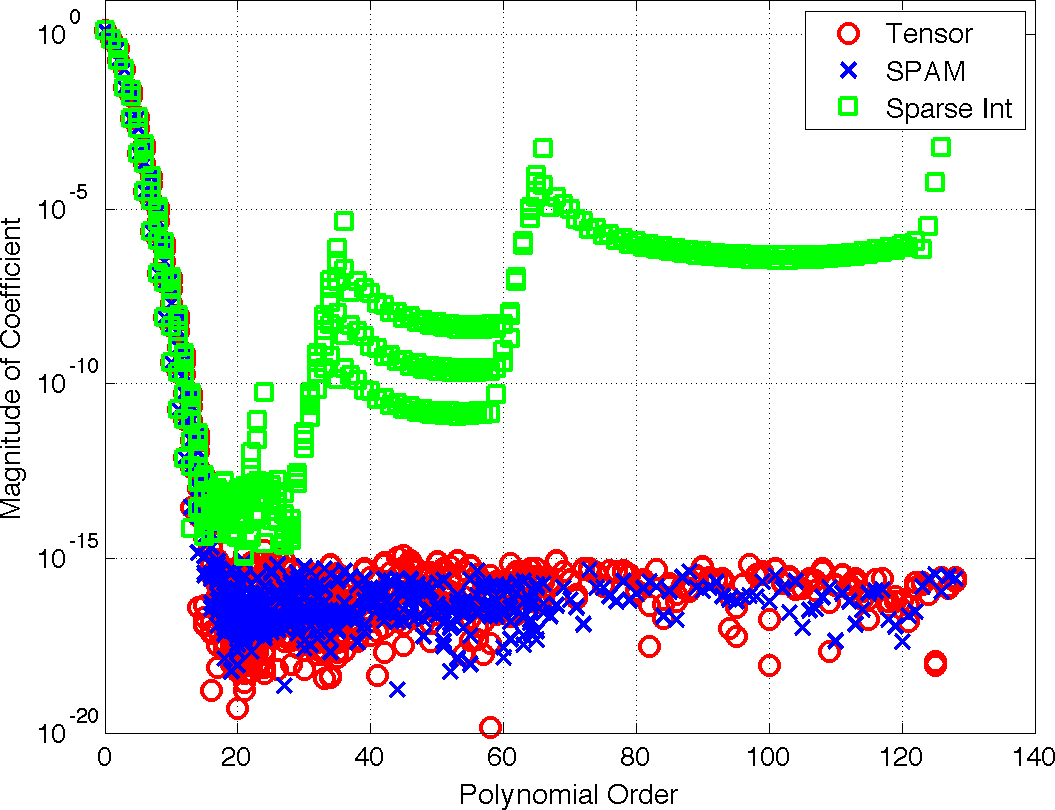}
}\\
\subfloat[SPAM]{
\includegraphics[scale=0.35]{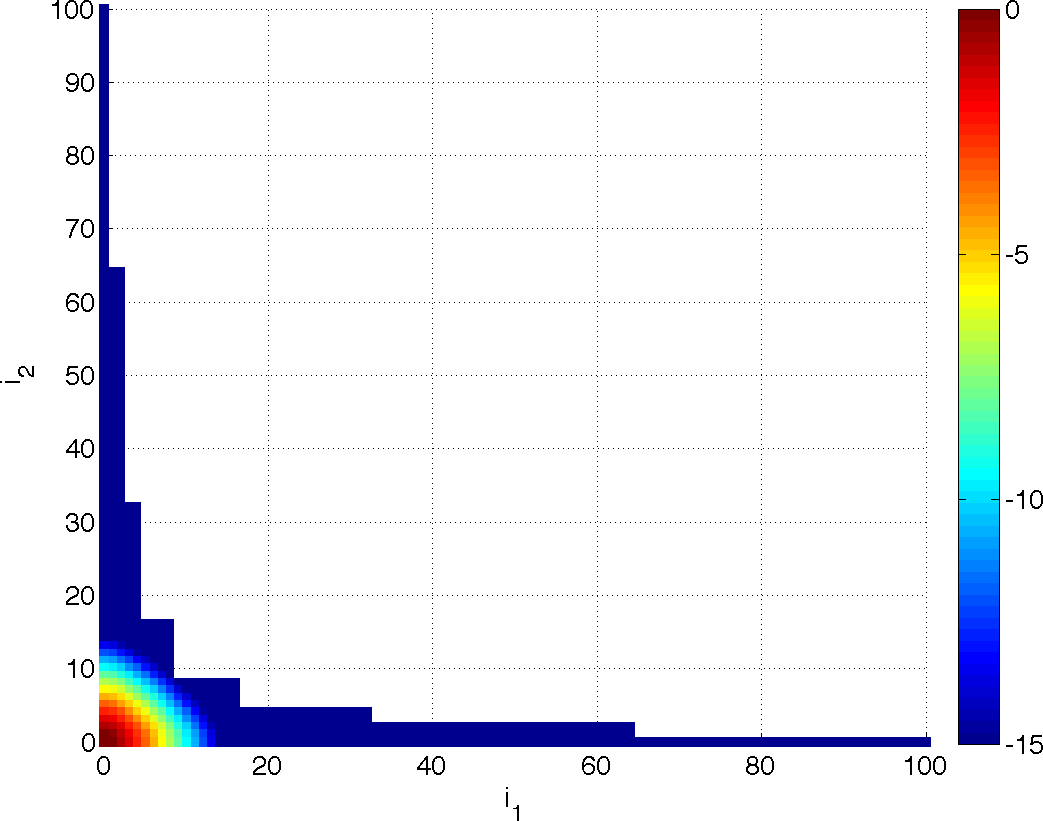}
}
\quad
\subfloat[Sparse grid integration]{
\includegraphics[scale=0.35]{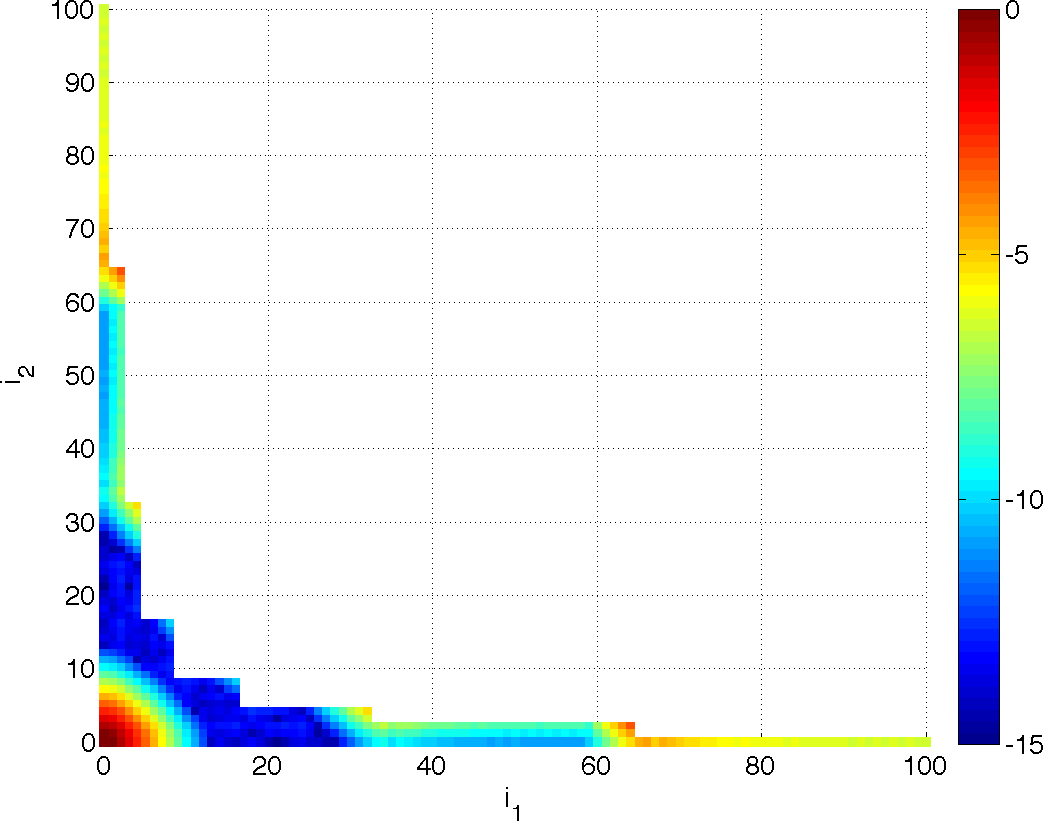}
}
\caption{Fourier coefficient approximations for $e^{s_1+s_2}$.}
\label{fig:bv2}
\end{figure}

\begin{figure}%
\centering
\subfloat[Tensor $100\times 100$]{
\includegraphics[scale=0.35]{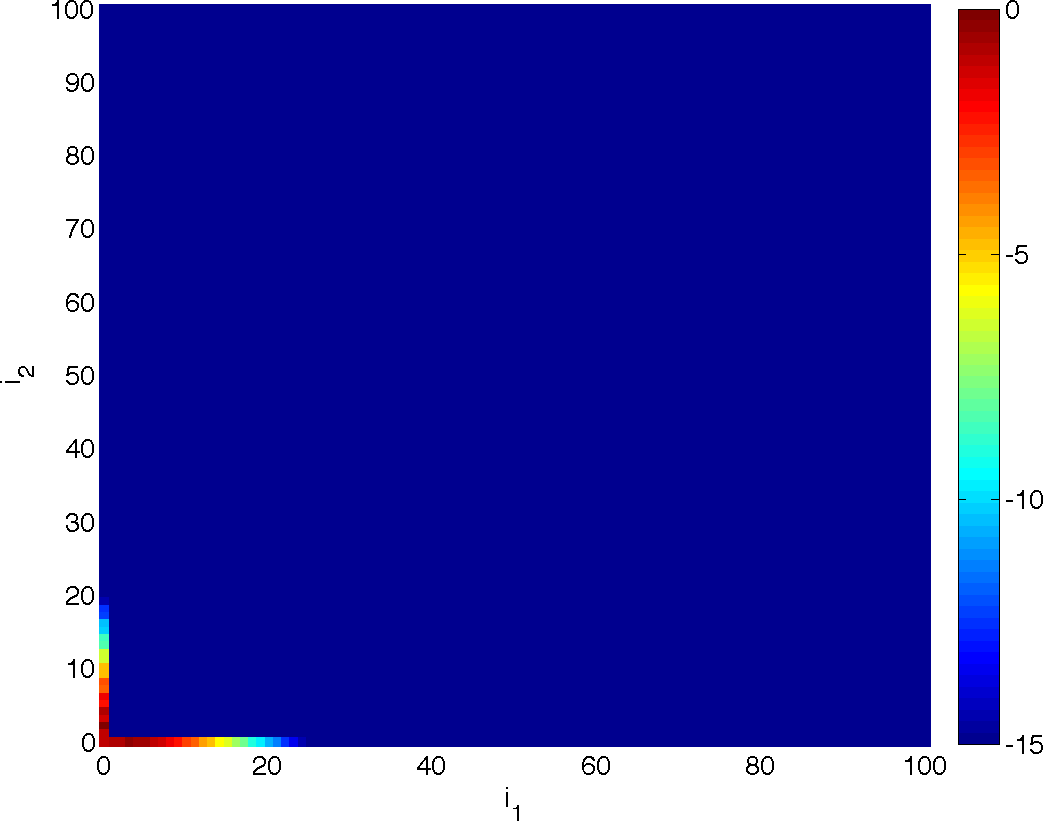}
}
\quad
\subfloat[Coefficient decay]{
\includegraphics[scale=0.35]{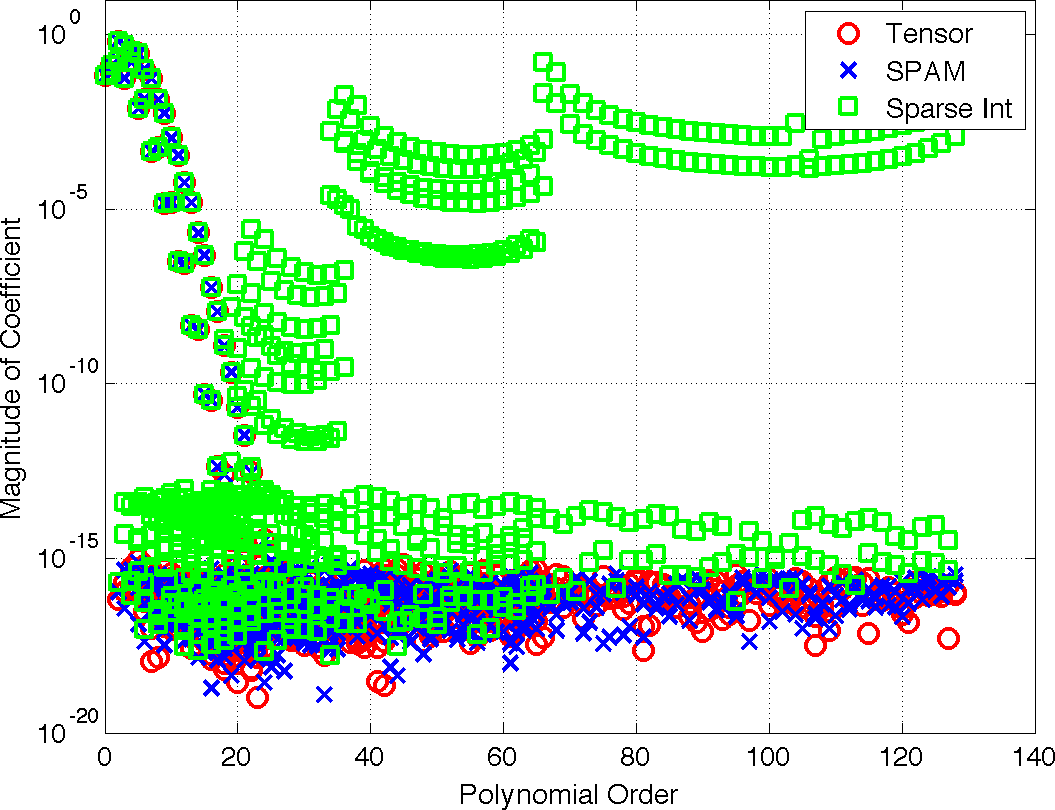}
}\\
\subfloat[SPAM]{
\includegraphics[scale=0.35]{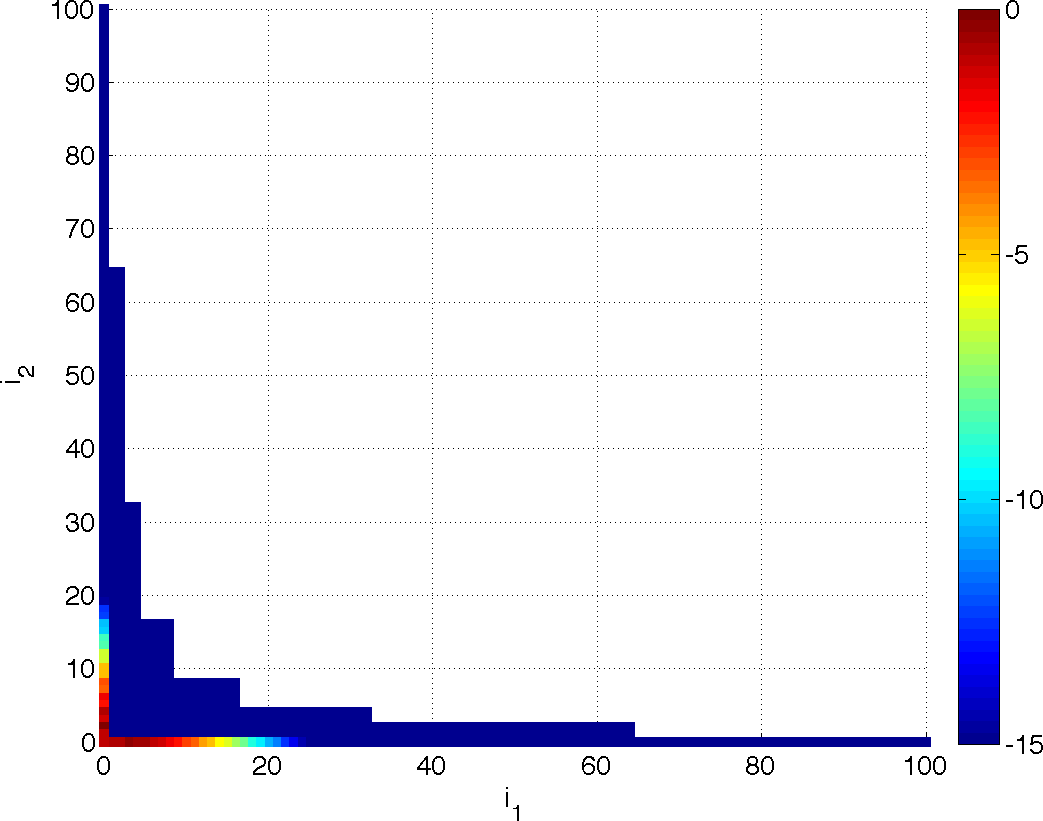}
}
\quad
\subfloat[Sparse grid integration]{
\includegraphics[scale=0.35]{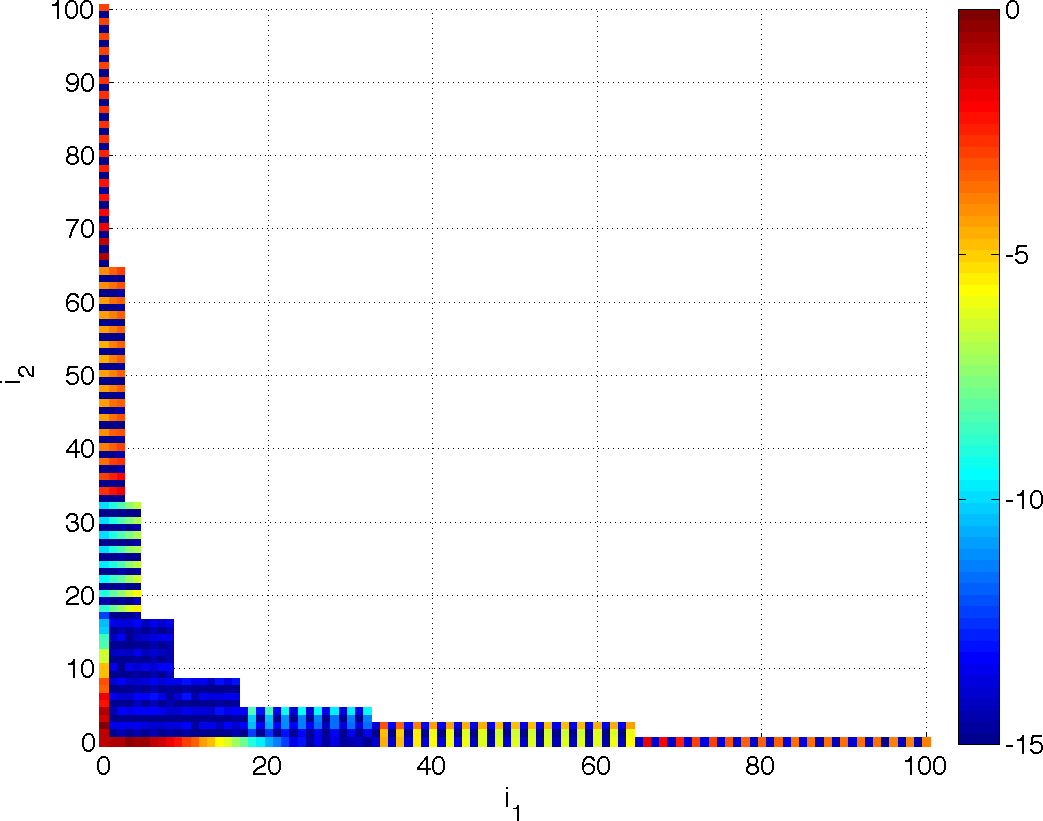}
}
\caption{Fourier coefficient approximations for $\sin(5(s_1-0.5)) + \cos(3(s_2-1))$.}
\label{fig:bv3}
\end{figure}

\begin{figure}%
\centering
\subfloat[Tensor $100\times 100$]{
\includegraphics[scale=0.35]{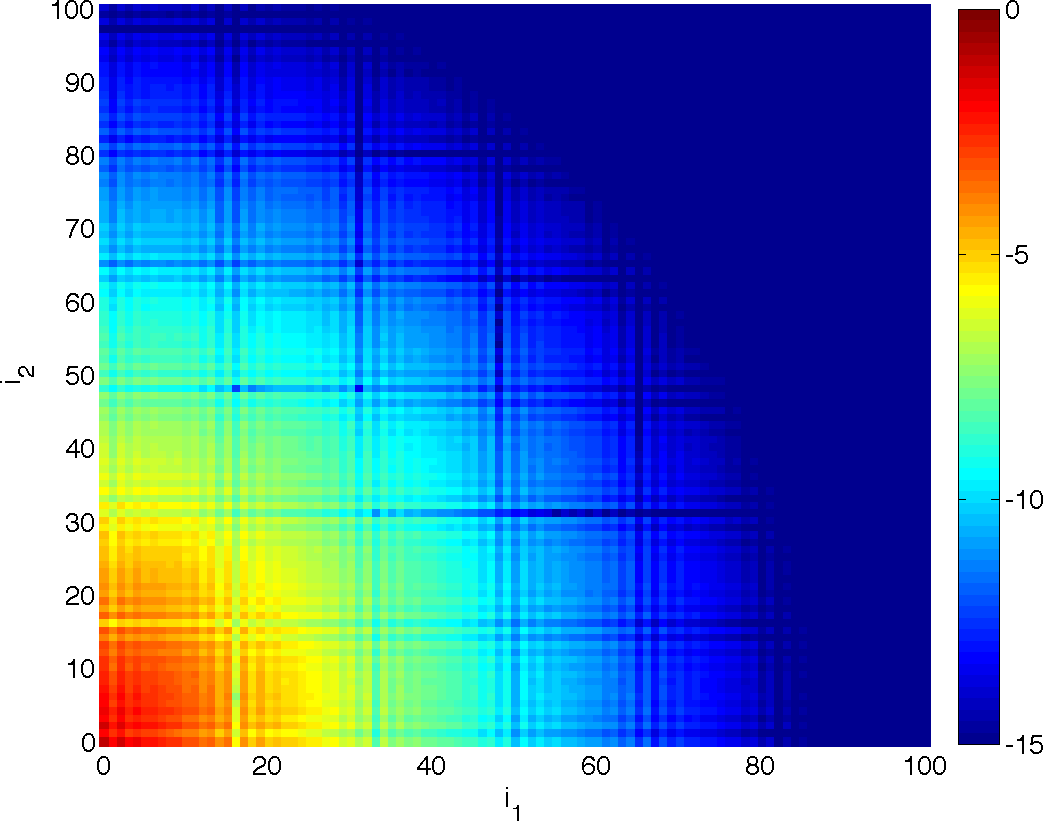}
}
\quad
\subfloat[Coefficient decay]{
\includegraphics[scale=0.35]{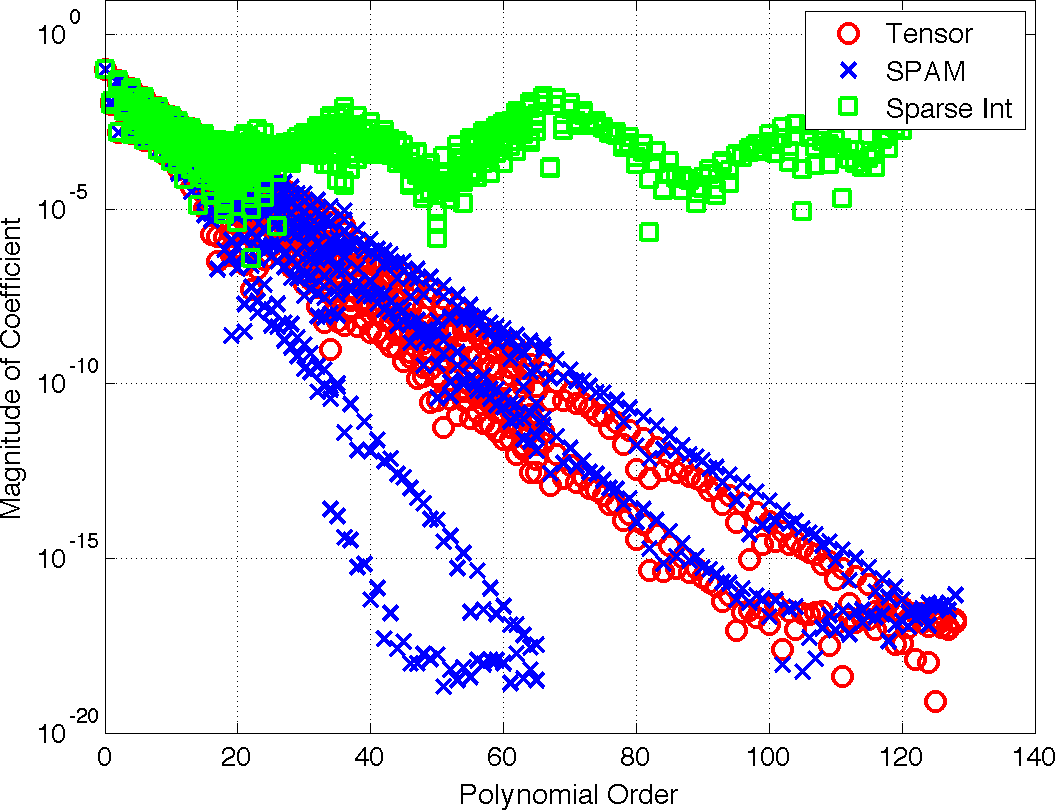}
}\\
\subfloat[SPAM]{
\includegraphics[scale=0.35]{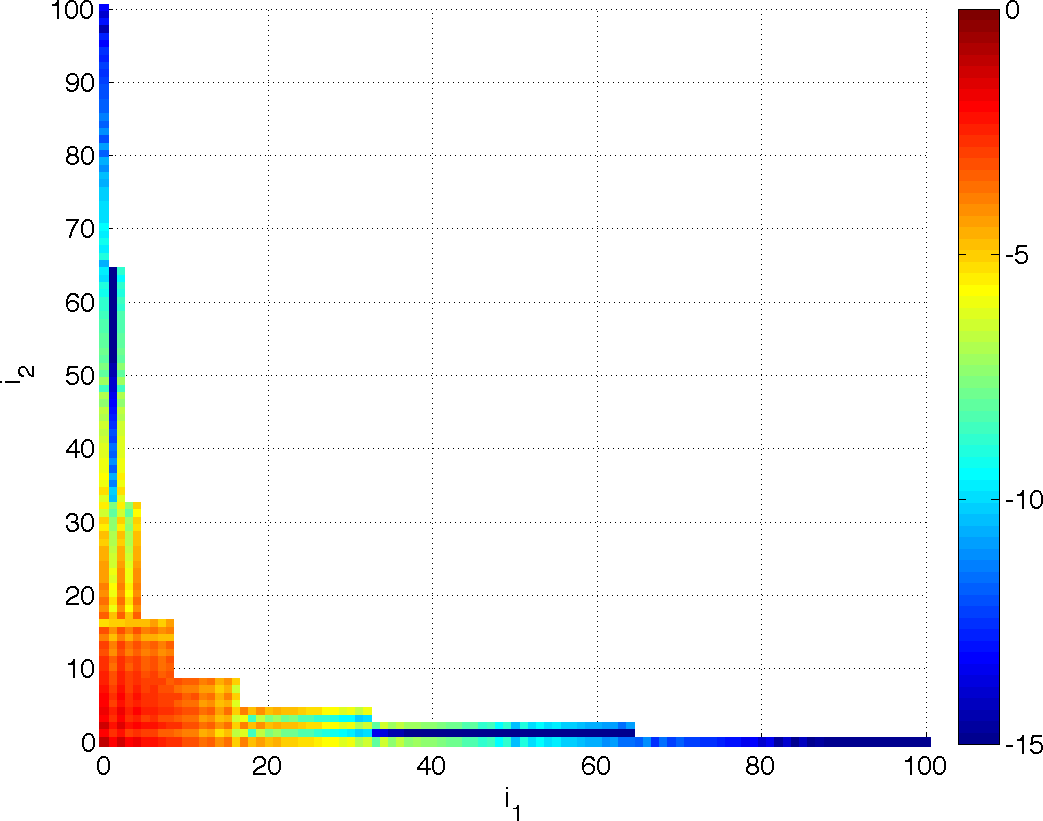}
}
\quad
\subfloat[Sparse grid integration]{
\includegraphics[scale=0.35]{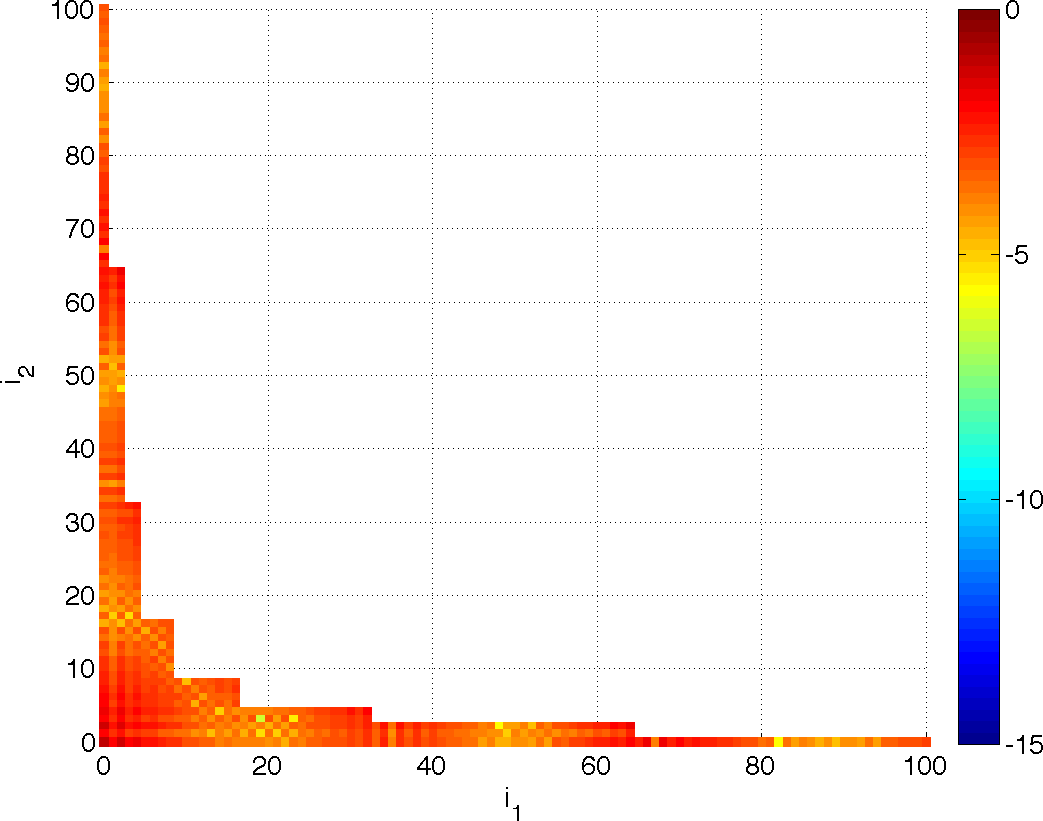}
}
\caption{Fourier coefficient approximations for $1/(2+16(s_1-0.1)^2+25(s_2+0.1)^2)$.}
\label{fig:bv4}
\end{figure}

\begin{figure}%
\centering
\subfloat[Tensor $100\times 100$]{
\includegraphics[scale=0.35]{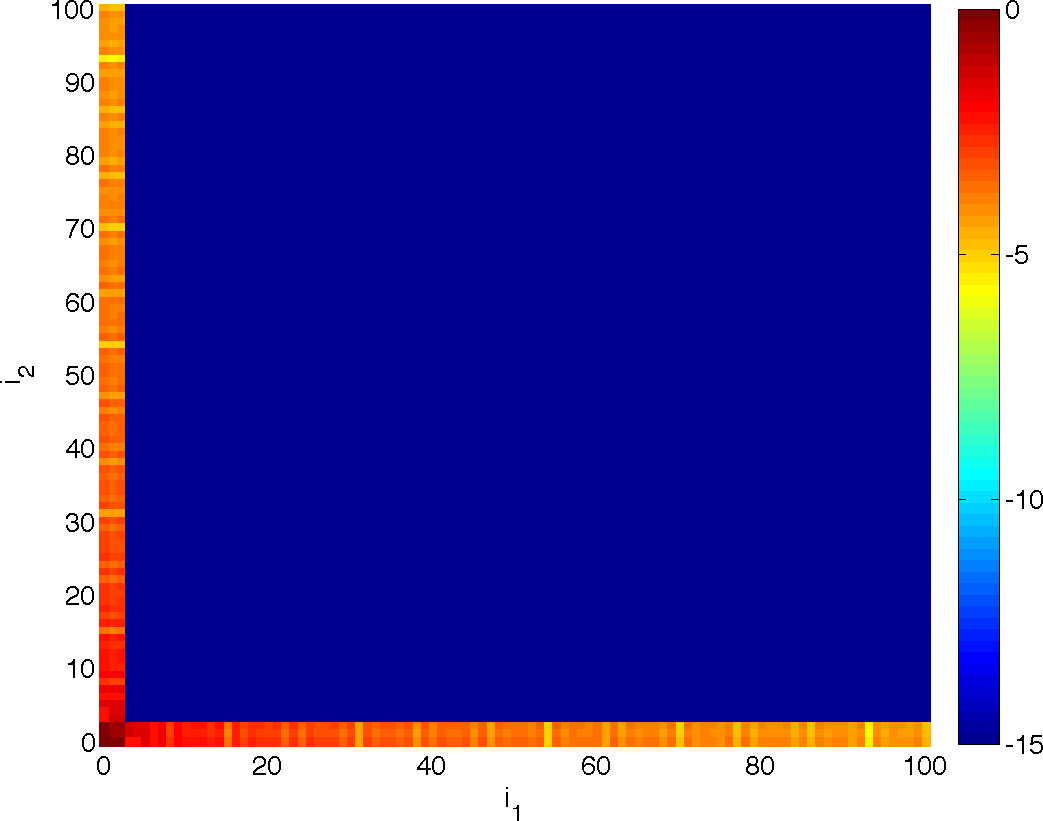}
}
\quad
\subfloat[Coefficient decay]{
\includegraphics[scale=0.35]{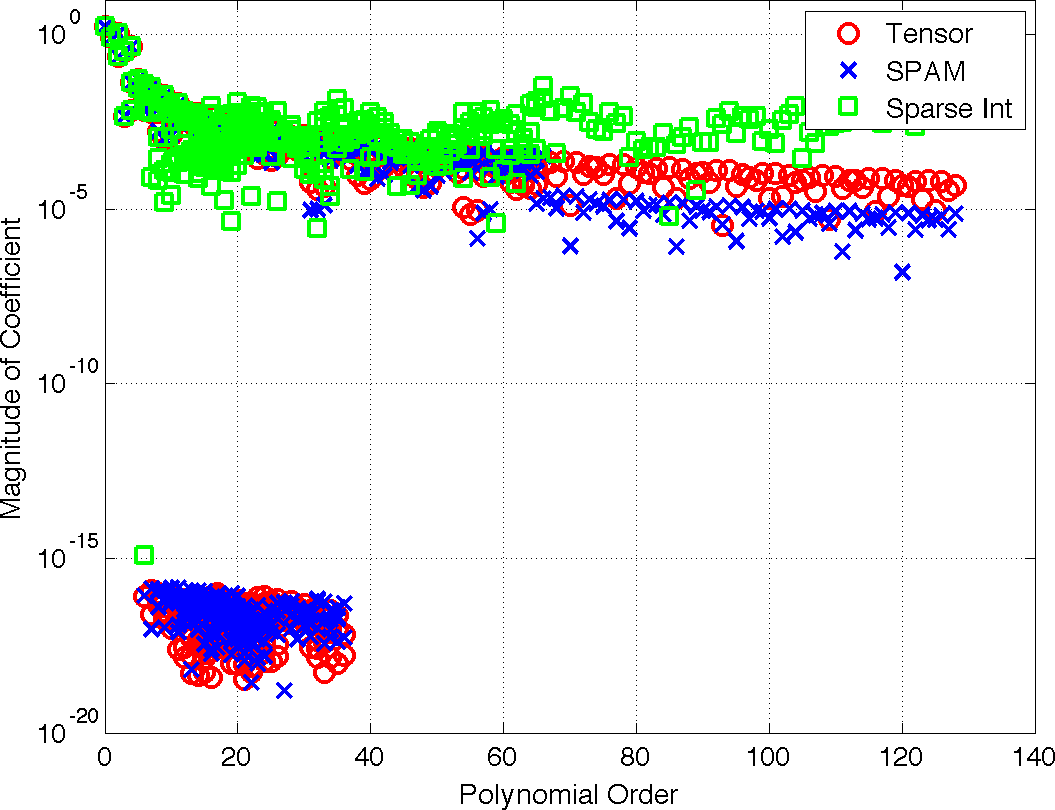}
}\\
\subfloat[SPAM]{
\includegraphics[scale=0.35]{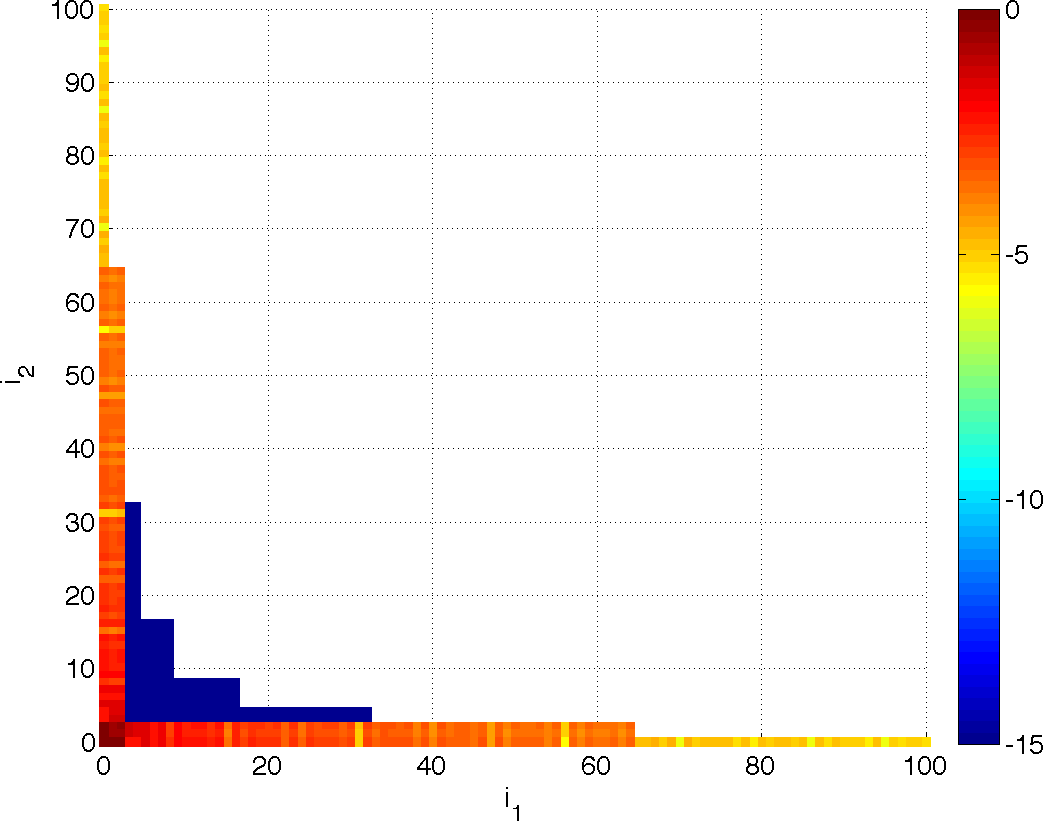}
}
\quad
\subfloat[Sparse grid integration]{
\includegraphics[scale=0.35]{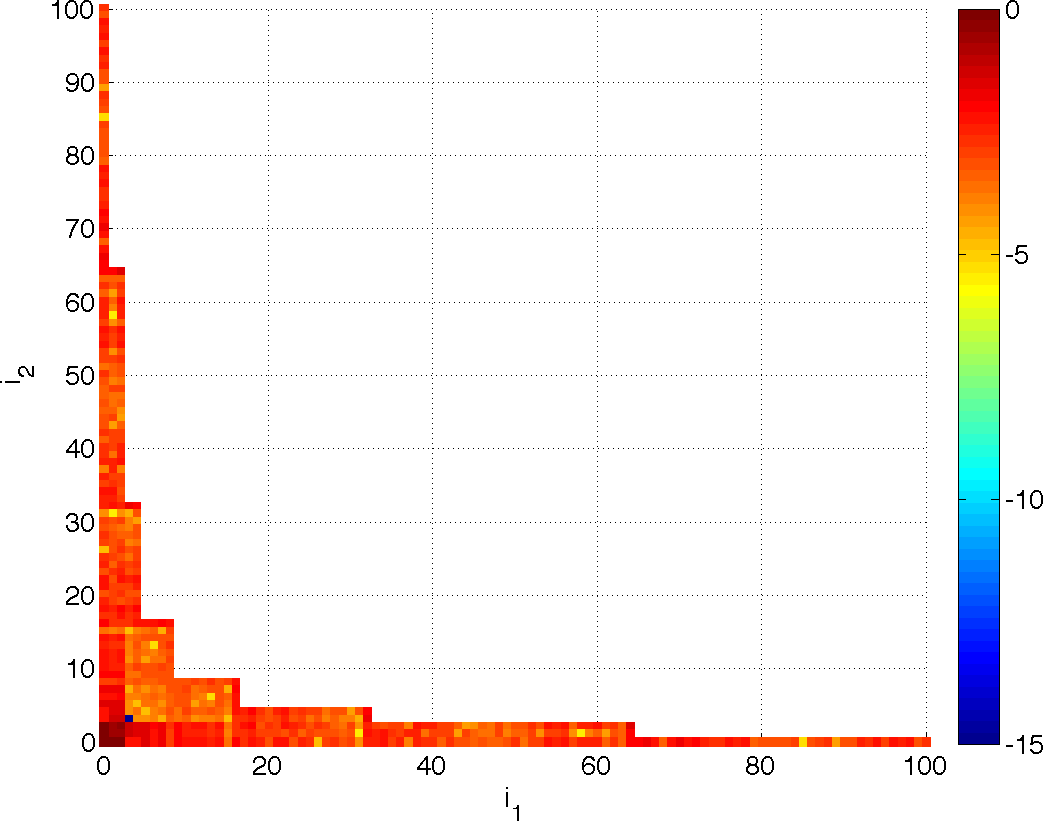}
}
\caption{Fourier coefficient approximations for $(|s_1-0.2|+|s_2+0.2|)^3$.}
\label{fig:bv5}
\end{figure}

\begin{figure}%
\centering
\subfloat[SPAM coefficient error]{\label{fig:err-a}
\includegraphics[scale=0.35]{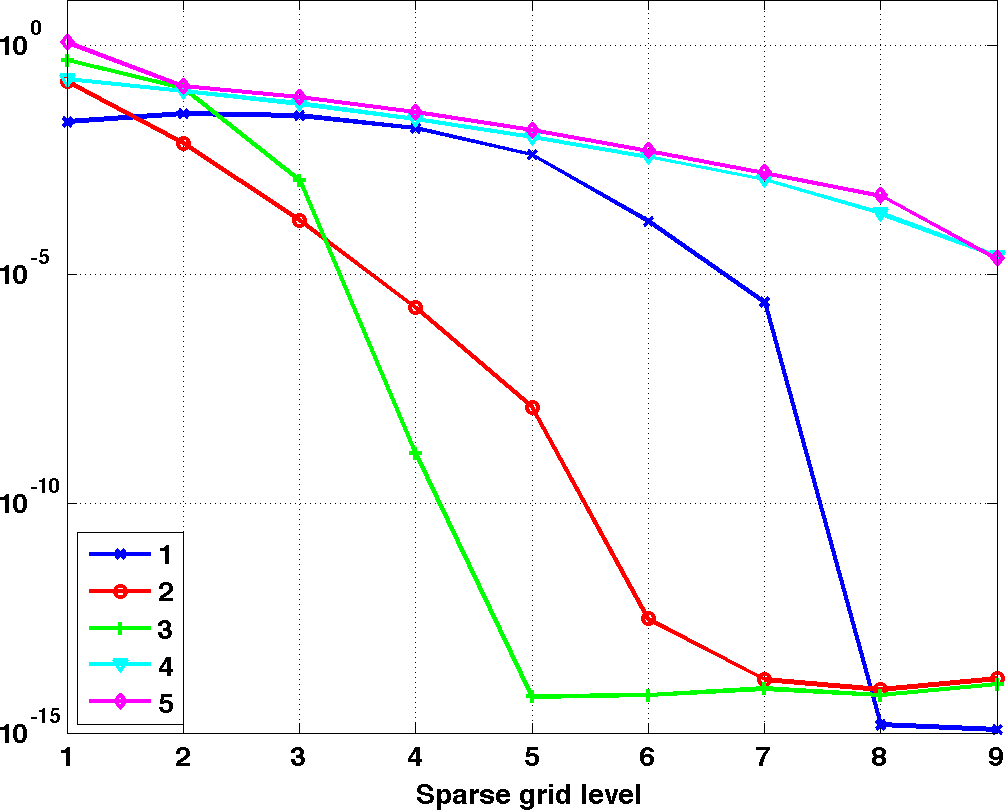}
}\quad
\subfloat[Sparse grid integration coefficient error]{\label{fig:err-b}
\includegraphics[scale=0.35]{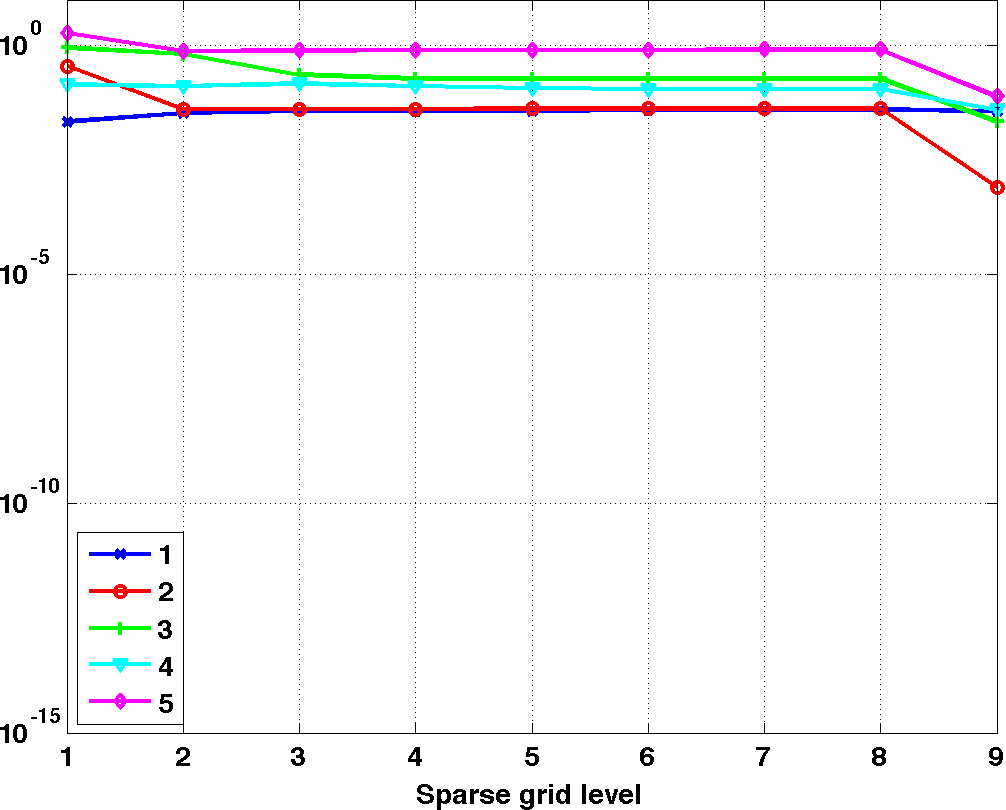}
}\\
\subfloat[Truncation error]{\label{fig:err-c}
\includegraphics[scale=0.35]{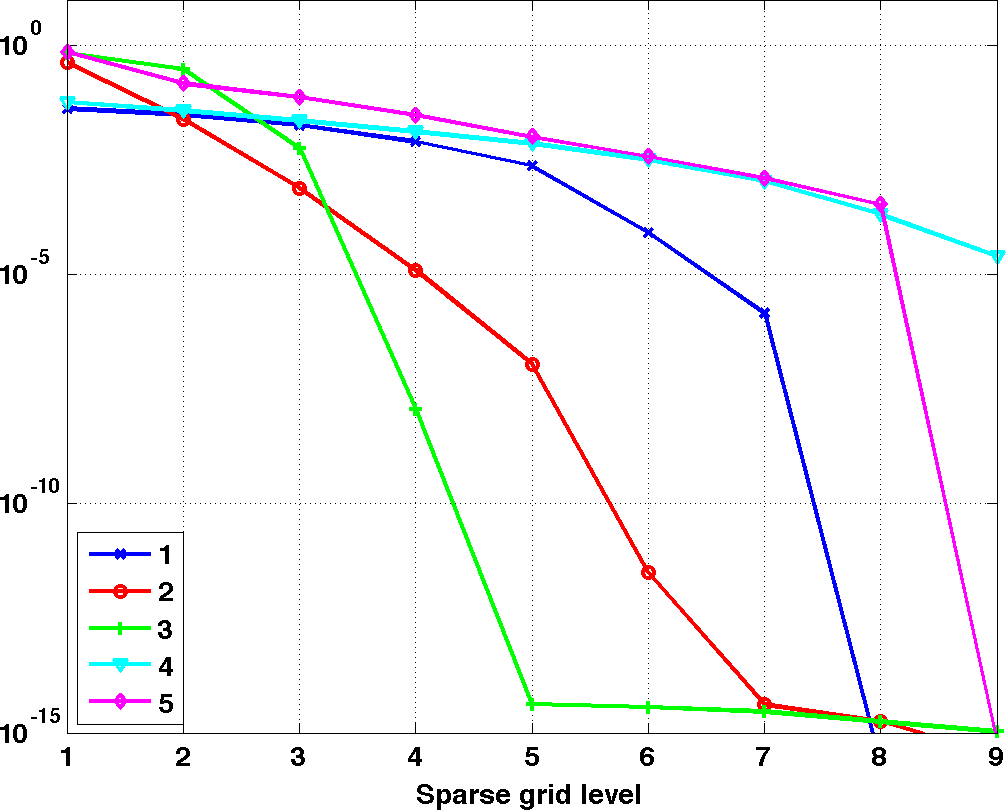}
}
\caption{Comparison of truncation error to error coefficient approximation between SPAM and direct sparse grid
integration for each of the five test functions, numbered according to Table \ref{tab:testfuncs}.}
\label{fig:err}
\end{figure}

\subsection{PDE with Random Input Data}
The last numerical example we examine is a linear elliptic
diffusion equation with a stochastic diffusion coefficient.  Let $D =
[0,1]\times[0,1]$ and $(\Omega,\mathcal{B},P)$ be a complete
probability space.  We seek the function
$u:D\times\Omega\rightarrow\mathbb{R}$ such that the following holds $P$-a.e.:
\begin{equation} \label{eq:pde}
\begin{split}
 -\nabla\cdot(a(x,\omega)\nabla u(x,\omega)) &= 1, \quad x\in D, \\
  u(x) &= 0, \quad x\in\partial D.
\end{split}
\end{equation}
Instead of the whole solution $u$, we are interested in computing the
response function
\begin{equation}
  g(\omega) = \int_D u(x,\omega) dx
\end{equation}
which is the spatial mean of $u$ over $D$.

The diffusion coefficient $a(x,\omega)$ is modeled as a random
field with exponential correlation:
\begin{equation}
  C(x,y) \equiv E[a(x,\omega)a(y,\omega)] = \sigma^2 e^{\|x-y\|_1/L}
\end{equation}
where $\sigma=0.1$ is the standard deviation of the field and $L=1$ is
the correlation length.  It is approximated through a truncated
Karhunen-Lo\'{e}ve expansion~\cite{Loeve78}:
\begin{equation}
  a(x,\omega) \;\approx\; \hat{a}_d(x,s(\omega)) \;=\; a_0(x) \,+\, \sum_{k=1}^d
  \sqrt{\lambda_k} a_k(x)\,s_k(\omega),
\end{equation}
where $a_0(x) = \mu = 0.2$ is the mean of the random field,
$(\lambda_k,a_k(x))$, $k=1,\dots,d$ are eigenvalue-eigenfunction pairs for the
covariance operator:
\begin{equation}
  \int_D C(x,y)a_k(x) dx = \lambda_k a_k(y), \quad y\in D,
\end{equation}
and $s = (s_1,\dots,s_d)$ are uncorrelated,
uniform random variables on $[-1,1]$.  We make the further modeling assumption
that the random variables are independent.  Define $\Gamma = [-1,1]^d$
to be the range of $s$ and 
\begin{equation}
w(s) = \left\{
\begin{array}{cl}
1/2^d & s\in[-1,1]^d\\
0 & \mbox{ otherwise}
\end{array}
\right.
\end{equation}
to be the density of $s$.  The eigenvalues and
eigenfunctions are computed using a pseudo-analytic procedure
described in~\cite{Ghanem91}.  The eigenvalues are sorted
in decreasing order, and we use the first $d=5$
eigenvalues/eigenfunctions to approximate the random field.

Let $\pi_i:[-1,1]\rightarrow\mathbb{R}$, $i=1,2,\dots$ be the normalized
Legendre polynomial of order $i-1$.  For a given multi-index
$\mathbf{i}=(i_1,\dots,i_d)$, define the tensor product polynomial
\begin{equation}
  \pi_{\mathbf{i}}(s) = \pi_{i_1}(s_1)\dots\pi_{i_d}(s_d).
\end{equation}  
Given a set $\sI$ of multi-indices,  we approximate $g(\omega)$ by
\begin{equation} \label{eq:g_pce}
  \hat{g}(\omega) = \sum_{\mathbf{i}\in \sI}
  \hat{g}_{\mathbf{i}}\pi_{\mathbf{i}}(s(\omega))
\end{equation}
where the unknown coefficients $\hat{g}_{\mathbf{i}}$ are computed
through pseudospectral projection using both SPAM
and sparse grid integration.  For a given $s$, the
corresponding response $g$ is computed by solving
\begin{equation} \label{eq:pde_sample}
\begin{split}
 -\nabla\cdot(\hat{a}_d(x,s)\nabla u(x) &= 1, \quad x\in D, \\
  u(x) &= 0, \quad x\in\partial D, \\
  g &= \int_D u(x) dx.
\end{split}
\end{equation}
These equations are discretized using piecewise linear finite
elements over quadrilateral mesh cells of size $1/512$, which gave a spatial error of $\mathcal{O}(10^{-6})$.  The
resulting linear algebraic equations are solved via preconditioned GMRES using
an algebraic multigrid preconditioner with tolerance of $10^{-12}$.
The finite element equations were implemented and solved using a
variety of packages within the Trilinos solver framework~\cite{TrilinosTOMS}.  
The resulting SPAM and sparse grid
integrations were provided by the Dakota package~\cite{DAKOTA}.  

Note that instead of using the growth relationship in \eqref{eq:growth}, we choose 
\begin{equation}
\label{eq:lingro}
n_m=2m-1,\quad m\geq 1.
\end{equation}
This growth relationship yields tensor grids with many fewer points compared to \eqref{eq:growth}. The multiplication
factor $2$ ensures that all tensor grids will share the mid-point of the domain, which reduces the total number of
function evaluations.
The corresponding coefficients of the
stochastic response function $g$ are plotted in Figure~\ref{fig:phys} by the degree of the corresponding multivariate
polynomial. The level parameter for the sparse grid is 4.

One can see
as the order of the polynomials increases, the coefficients generated
by SPAM decay as they should, whereas those generated through direct
sparse integration begin to diverge for the higher order polynomials.
Note, however, that the difference is not as pronounced compared to the bivariate test cases. We attribute this to the
use of the growth relationship \eqref{eq:lingro}, as opposed to
\eqref{eq:growth} used with the bivariate functions. 

\begin{figure}%
\centering
\includegraphics[scale=0.6]{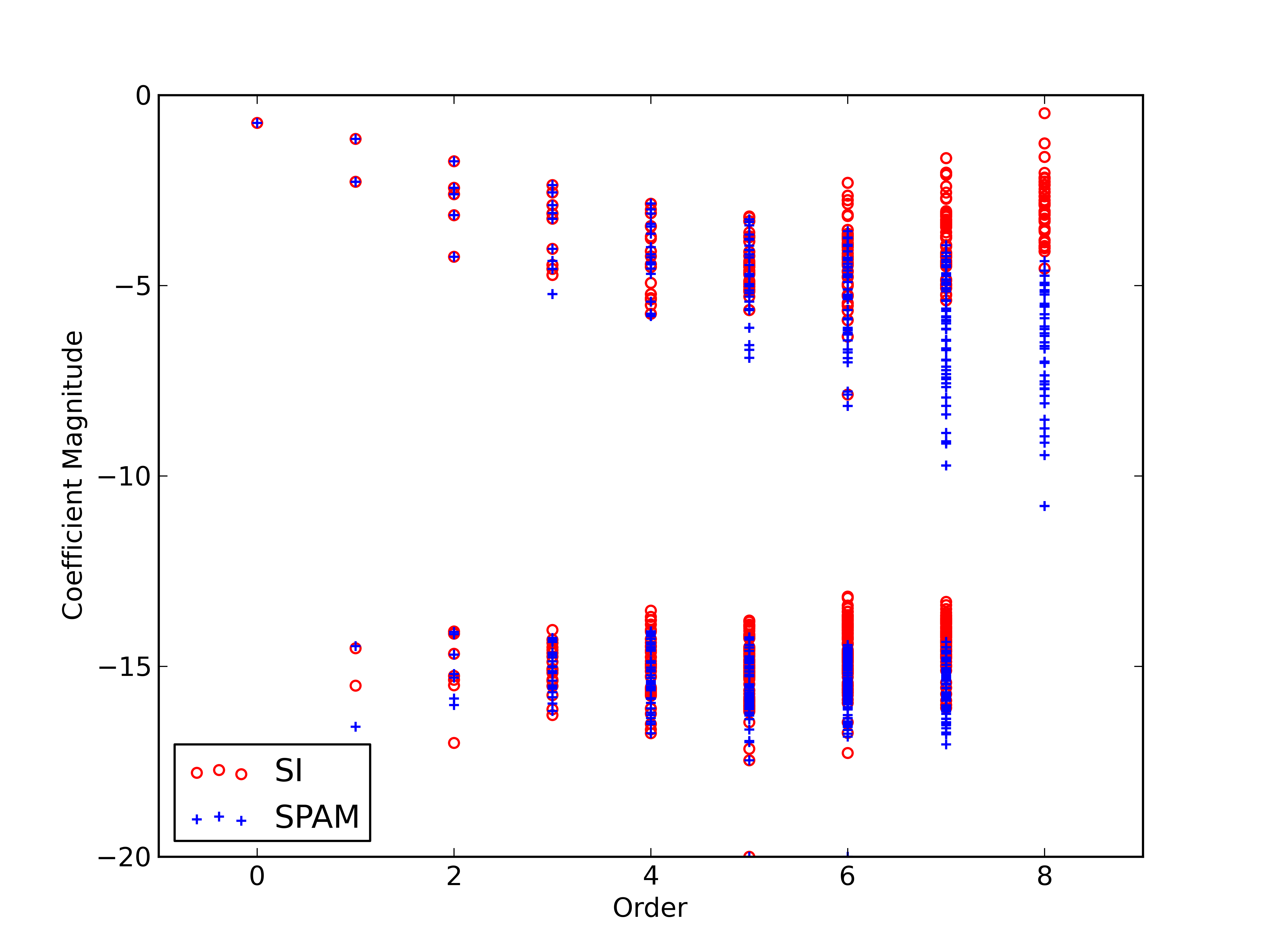}
\caption{Comparison of approximate Fourier coefficients of the
  stochastic response~\eqref{eq:g_pce} of the linear diffusion
  problem~\eqref{eq:pde} using SPAM and sparse integration for dimension $d=5$ and level $l=4$. The
figure plots the coefficients according to the degree of associated polynomial.}
\label{fig:phys}
\end{figure}

\section{Conclusions}
\label{sec:conclusions}

Sparse grid integration rules are constructed as linear combinations of tensor product quadrature rules. By taking
advantage of the equivalence between the tensor product Lagrange interpolant and a pseudospectral approximation with a
tensor product orthogonal polynomial basis, we can linearly combine the tensor product polynomial expansions associated
with each tensor grid quadrature rule to produce a sparse pseudospectral approximation. We have numerically compared
this approach to direct sparse grid integration of the Fourier coefficients. The experiments show
convincingly that the direct integration approach produces inaccurate approximations of the Fourier coefficients
associated with the higher order polynomial basis functions, while the SPAM coefficients are much more accurate.

The difference between SPAM and the sparse grid integration of the Fourier coefficients is present in all Smolyak type
approximations -- including anisotropic and adaptive variants. While not presented explicitly in this paper due to space
limitations, the authors have conducted similar studies on such variants with similar results. The conclusions are
clear. Given a function evaluated at the nodes of a sparse grid integration rule, the proper way to approximate the
Fourier coefficients of an orthogonal expansion is the SPAM.





\bibliographystyle{elsarticle-num}
\bibliography{paulconstantine}







\end{document}